\documentclass{article}
\usepackage{amsmath}
\usepackage{amsfonts}
\usepackage{amsthm}
\usepackage{amssymb}
\usepackage[english]{babel}
\usepackage[T1]{fontenc}

\theoremstyle{definition}

\newtheorem{defi}{Definition}[section]
\newtheorem{rem}[defi]{Remark}
\newtheorem{nota}[defi]{Notation}

\theoremstyle{plain}

\newtheorem{thm}[defi]{Theorem}
\newtheorem{cor}[defi]{Corollary}

\newtheorem{lem}[defi]{Lemma}
\newtheorem{prop}[defi]{Proposition}

\input xy
\xyoption{all}

\author{Marco Antei \and Michel Emsalem}

\title{Galois Closure of Essentially Finite Morphisms}
\begin{document}
\maketitle

\textbf{Abstract.} Let $X$ be a reduced connected $k$-scheme pointed at a rational point $x \in X(k)$. By using tannakian techniques we construct the Galois closure of an essentially finite $k$-morphism $f:Y\to X$ satisfying the condition $H^0(Y,\mathcal{O}_Y)=k$; this Galois closure is a torsor $p:\hat{X}_Y\to X$ dominating $f$ by an $X$-morphism $\lambda:\hat{X}_Y\to Y$ and universal for this property. Moreover we show that $\lambda:\hat{X}_Y\to Y$ is a torsor under some finite group scheme we describe. Furthermore we prove that the direct image of an essentially finite  vector bundle over $Y$ is still an essentially finite vector bundle over $X$. We develop for torsors and essentially finite morphisms a Galois correspondence similar to the usual one. As an application we show that for any pointed torsor $f:Y \to X$ under a finite group scheme satisfying the condition $H^0(Y,\mathcal{O}_Y)=k$, $Y$ has a fundamental group scheme $\pi _1 (Y,y)$ fitting in a short exact sequence with $\pi _1 (X,x)$. 

\tableofcontents
\bigskip

\section{Introduction}

Let $k$ be a field and $X$ a proper reduced $k$-scheme satisfying the condition $H^0 (X,\mathcal{O}_X)=k$ (then it is in particular  connected).\footnote{In Nori's papers $k$ is a perfect field. The only point where this hypothesis is used is to insure that the ring of endomorphisms $H^0 (X,\mathcal{O}_X)$ of the unit object $\mathcal{O}_X$ of the tannakian category $EF(X)$ is $k$. We do not assume here that $k$ is perfect, but we make the hypothesis $H^0 (X,\mathcal{O}_X)=k$} The Nori fundamental group scheme was defined as the Galois tannakian group scheme of the category of {\it essentially finite} vector bundles on $X$ \cite{Nor2},  Ch. I, \S 2.3. We define the notion of essentially finite morphisms: these are finite faithfully flat morphims $ f: Y \to X$ such that $f_\ast (\mathcal{O}_Y)$ is essentially finite. For instance torsors under finite group schemes, or more generally towers of torsors are essentially finite morphisms. On the other hand if $f:Y\to X$ is a finite \'etale morphism then it is  essentially finite; when $ch(k)=0$ the converse is also true: more precisely essentially finite morphism $f:Y\to X$ such that $H^0(Y,\mathcal{O}_Y)=k$ are finite \'etale morphisms (Remark \ref{RemEtalEssential}).

The aim of this article is to prove the existence of the Galois closure of an essentially finite morphism $f:Y\to X$. It is a torsor under a finite group scheme which dominates $f$ and satisfies an obvious universal property:

\begin{thm}\label{th1}  (see theorem \ref{universel}) Assume that $\theta: X \to Spec (k)$ is a proper reduced locally noetherian $k$-scheme over a field $k$ such that $H^0 (X,\mathcal{O}_X)=k$, endowed with a rational point $x: Spec (k) \to X$. Let $f: Y \to X$ be an essentially finite morphism such that  $H^0 (X, f_\ast (\mathcal{O}_Y))=k$ given with a rational point $y $ in the fiber $Y_x$ of $f$ at $x$. 
\begin{enumerate}
\item Under these conditions there exists a pointed torsor $p : \hat Y \to X$ under some $k$-group $G$ and a unique pointed morphism $\lambda: \hat Y \to Y$ satisfying $ f \circ \lambda = p$ ; moreover $\lambda$ is faithfully flat. The couple $( \hat Y, \lambda )$ is unique up to unique isomorphism of pointed torsors.

\item The Galois closure  $\hat Y$ satisfies the following universal property: for any quotient pointed torsor $T$ under a finite $k$-group $H$  and any pointed faithfully flat morphism $\mu: T \to Y$, there exists a unique morphism of pointed torsors $(u, \varphi )$ where $\varphi: H \to G$ and $u: T\to \hat Y$ making the following diagram commutative 

$$
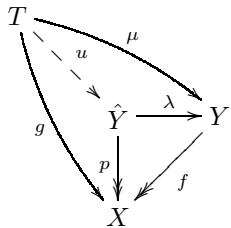
\xymatrix{T\ar@/_/[ddr]_{g} \ar@/^/[drr]^{\mu} \ar@{-->}[dr]^{u} & &  \\  & \hat Y \ar@{->>}[d]_{p}
\ar[r]^{\lambda}& Y\ar@{->>}[dl]^{f}\\ & X &   }$$
\end{enumerate}
\end{thm}

Moreover given a torsor $T \to X$ under a finite group scheme $G$, there is a Galois correspondence between subgroups of $G$ and intermediate essentially finite morphisms, correspondence similar to the usual Galois correspondence. Then we prove that the direct image of an essentially finite vector bundle over $Y$ is still an essentially finite vector bundle over $X$: 

\begin{thm}\label{th2} (see theorem \ref{teoDEDICATOaWEIL}) Let $X$  be a reduced, proper $k$-scheme such that $H^0 (X,\mathcal{O}_X)=k$, endowed with a rational point $x \in X(k)$. Let $f: Y \to X$ be an essentially finite morphism of $k$-schemes. Assume either $Y$ to be a reduced, proper $k$-scheme with $H^0 (Y,\mathcal{O}_Y)=k$, endowed with a rational point $y \in Y_x(k)$, or $f: Y \to X$ to be a torsor under a finite group scheme. Let $\mathcal{F}$ be a vector bundle over $Y$ trivialized by a torsor over $Y$ under a finite $k$-group scheme. Then the sheaf
$f_{\ast}(\mathcal{F})$ is an essentially finite vector bundle over $X$.
\end{thm}

This allows us to give a positive answer to a natural question: composite of two essentially finite morphisms is essentially finite. 
The point of view chosen here is to use the tannakian techniques introduced by Nori in his definition of the fundamental group scheme. Another approach discussed also by Nori is to see the fundamental group scheme as the projective limit of the finite group schemes occurring in torsors over $X$. This last point of view has been generalized by Gasbarri \cite{Gas}, \S 2 allowing him to define the fundamental group schemes of relative schemes over Dedekind schemes. 

We show that if $f : Y \to X$ is a torsor under a finite group scheme satisfying the condition $H^0 ( Y , \mathcal {O} _Y)=k$ then $Y$ has a fundamental group scheme which fits in a short exact sequence of fundamental group schemes (theorem \ref{suite-exacte}). Using Gasbarri's approach M. Garuti addressed the question of the Galois closure of a tower of torsors \cite{Gar}, Theorem 1. His Galois closure is not, however, the smallest possible (Remark \ref{remGarou}), but his construction is essential in our proof of theorem \ref{teoDEDICATOaWEIL}. Moreover he obtained a short  exact sequence similar to ours in the more general setting of relative schemes over Dedekind schemes (\cite{Gar}, Theorem 3).

We limited ourselves to the case of a proper base scheme $X$. Note that in the case of an affine curve $X$, N. Borne defines in \cite{Bor} the tame fundamental group scheme of $X$ using the tannakian category of parabolic vector bundle. Independently H.Esnault and P.H.Hai define the fundamental group scheme of $X$ in the characteristic $0$ case, using the tannakian category of vector bundles with connexion \cite{EsnHai1}. In the context of these theories, the construction of the Galois closure make sense and could be performed following the same lines.

\section{The fundamental group scheme.}
\label{sez:Intro}
In this section we briefly recall the tannakian
construction of the Nori fundamental group scheme (cf. \cite{Nor} and
\cite{Nor2}). The situation is the following:

\begin{nota}\label{notazCAP4}
From now on $k$ will denote a field.  Let $\theta: X \to Spec (k)$ be a  reduced and proper  $k$-scheme such that $H^0( X,\mathcal{O}_X)=k$ endowed with a section
$x:Spec(k)\rightarrow X$. 

We will consider in this article {\it torsors} over $X$ under finite $k$-group schemes $G$. They are finite, faithfully flat and $G$-invariant morphisms $p: Y \to X$, locally trivial for the fpqc topology, where $Y$ is a $k$-scheme endowed with a right action of $G$ and $X$ with the trivial action. 
\end{nota}

\medskip
Nori defines in \cite{Nor2},  Ch. I, \S 3 the fundamental group scheme of $X$ as the Galois group of a neutral tannakian category generated by the so called ``finite vector bundles''.

\begin{defi}\label{defFibratoFinito} A vector bundle $\mathcal{F}$ over $X$ is
said to be \textit{finite}  if there exist two polynomials $
p(x),q(x) \in \mathbb{Z}[x]$, where $p(x)\neq q(x)$ have
   nonnegative coefficients
    such that $$p(\mathcal{F})\simeq
   q(\mathcal{F}),$$ where the sum is the direct sum of vector
   bundles over $X$ and the product is the tensor product over $\mathcal{O}_X$.
\end{defi}

It is a fact that  finite vector bundles are {\it semi-stable} in the sense of Nori (\cite{Nor2}, Ch. I, \S 2.3) and Nori considers the subcategory generated by the finite vector bundles in the tannakian category ${{SS}} (X)$ of semi-stable vector bundles on $X$. More precisely:

\begin{defi}\label{defCATEGORIA.GENERATA}
Let $S$ be a set of objects of ${{SS}} (X)$. We denote ${{SS}} (X)(S)$ the full abelian subcategory of ${{SS}} (X)$ generated by $S$.
Since finite vector bundles over $X$ are semi-stable (\cite{Nor}
Corollary 3.5), if $F$ is the set of finite vector bundles over $X$, we denote $EF(X)={{SS}} (X)(F)$. Objects of $EF(X)$ are called essentially finite vector bundles over $X$.

More generally take any subset $U\subseteq Ob(EF(X))$ and let $U^{\vee}$ be the
set of duals of objects of $U$; let $U_1:=U\cup U^{\vee}$ and $S$
be the set of all possible tensor products of elements in $U_1$. We
denote $SS(X)(S)$ by $EF(X,U)$ the full abelian subcategory of
$SS(X)$ \textit{generated} by $S$. Let
\begin{displaymath}
\begin{array}{cc}
  i_U: & EF(X,U)\rightarrow \mathcal{Q}coh(X)
\end{array}
\end{displaymath}
be the inclusion functor and let
\begin{displaymath}
\begin{array}{cc}
  x^{\ast}: & \mathcal{Q}coh(X)\rightarrow k\text{-mod}, \\
   & \mathcal{F}\mapsto \mathcal{F}_{x}.
\end{array}
\end{displaymath}
be the fiber functor associated to the section $x\in X(k)$.

In what follows, we will use the neutral fiber functor $$\omega _U = x^\ast \circ i_U: EF(X,U) \to k\text {-mod}.$$ When $EF(X,U) = EF(X)$, we will use the notation $\omega = \omega _U$ and $ i_X= i_U$. 
\end{defi}
\medskip
Nori proves the following result (\cite{Nor2}, Ch. I, Proposition 3.7)

\medskip
\begin{thm}\label{proposizioneESSENZFINITTANNAKIAN}
For any $U\in Ob(EF(X))$, the category
$(EF(X,U),\otimes,\omega_U,\mathcal{O}_X)$ is a neutral tannakian category.\end{thm}  

\medskip

\indent Let $(\mathcal{C},\otimes,\gamma,1_{\mathcal{C}})$ be any neutral tannakian category, where $\gamma $ is a neutral fiber functor, let $S$ be any $k$-scheme and
  $\alpha, \beta:\mathcal{C}\to \mathcal{Q}coh(S)$ be two fiber functors with values in the category of quasi-coherent sheaves on $S$. Consider
  the functor 
$$\begin{array}{rrcl}
\underline{Isom}_S^{\otimes}(\alpha,\beta):& S\text{-}\mathcal{S}ch & \to & \mathcal{S}et\\
 & T & \mapsto & Isom^{\otimes}_T(\varphi^{\ast}\circ
\alpha,\varphi^{\ast}\circ \beta)
\end{array}$$ where $\varphi:T\to S$ is a morphism of schemes, $S\text{-}\mathcal{S}ch$ is the category of relative schemes over
$S$, $\mathcal{S}et$ is the category of sets and
$Isom^{\otimes}(\cdot,\cdot)$ stands for (iso)morphisms commuting
with the tensor product.

 General Tannaka duality insures that $\underline{Aut}^{\otimes}_k(\gamma):=\underline{Isom}^{\otimes}_k(\gamma,\gamma)$
is represented by an affine group scheme and that the category $\mathcal C$ is equivalent to the category of representations of this group scheme (cf. for instance \cite{DelMil}, Theorem 2.11). We will refer to $G:=\underline{Aut}^{\otimes}_k(\gamma)$ as the
tannakian Galois group scheme of the tannakian category
$\mathcal{C}$ attached to the neutral fiber functor $\gamma $.

Let $p:S\to Spec(k)$ be a scheme and $Fib_{S}(\mathcal{C})$ the
  category of fiber functors $\mathcal{C}\to \mathcal{Q}coh(S)$.
  Let $G\text{-}\text{Tors}_S$ be the category of (right) $G$-torsors over
  $S$. We have the following fundamental result (\cite{DelMil}, Theorem 3.2 and \cite{Del} \S 3):
\medskip

\begin{thm}\label{teorDELIGNEMILNE} The functor \begin{displaymath}
\begin{array}{rcl}
  Fib_{S}(\mathcal{C}) & \rightarrow & G\text{-}\text{Tors}_S \\
  \eta & \mapsto & \underline{Isom}^{\otimes}_S(p^{\ast}\circ \gamma,\eta)
\end{array}
\end{displaymath}
   is an equivalence of gerbes. \end{thm}

\medskip

\begin{defi}Let $k$ be a field, $\theta : X \to Spec (k)$ a reduced and proper $k$-scheme such that $H^0 (X,\mathcal{O}_X)=k$ and let $x\in X(k)$. Let $EF(X)$ be the
category of essentially finite vector bundles over $X$ and
 $\omega:=x^{\ast}\circ i_X:EF(X)\to k$-$mod$ the fiber functor defined above. Then $\pi_1(X,x):=\underline{Aut}^{\otimes}_k(\omega)=\underline{Isom}^{\otimes}_k(\omega,\omega)$ is the fundamental group
  scheme of $X$ in $x$. \end{defi}

The natural inclusion fiber functor $i_U:EF(X,U)\to
\mathcal{Q}coh(X)$ gives rise to a torsor over $X$ given by
$$\hat X_U:=\underline{Isom}^{\otimes}_X(\theta^{\ast}\circ
\omega_U,i_U)$$ under the (right) action of the affine $k$-group
scheme $\pi_1(X,U,x):=\underline{Aut}^{\otimes}_k(\omega_U)$.
Moreover the fiber at $x$ is equipped with a rational point $\hat x _U$. Indeed we have the following canonical isomorphisms (cf. lemma \ref{2}) $$x^\ast (\hat X_U)\simeq x^\ast \underline{Isom}^{\otimes}_X(\theta^{\ast}\circ
\omega_U,i_U)\simeq$$$$\simeq  \underline{Isom}^{\otimes}_k(x^\ast \circ  \theta^{\ast}\circ
\omega_U,x^\ast \circ i_U)\simeq \underline{Isom}^{\otimes}_k(\omega_U,\omega _U)=\pi _1 (X,U,x)$$

The image of the neutral element of $\pi _1 (X,U,x)$ in $x^\ast (\hat X_U)$ is $\hat x _U$ by definition. A torsor over $X$ under a finite $k$-group scheme endowed with a $k$-rational point in the fiber of $X$ will be referred to as a pointed torsor.

\begin{defi}\label{torseur-universel} We
will call $(\hat X_U, \hat x_U)$ the universal $\pi_1(X,U,x)$-torsor over
$X$ (associated to the tannakian category $EF(X,U)$). When $EF(X,U)=EF(X)$, the corresponding universal torsor will be denoted $(\hat X, \hat x)$.
\end{defi}

\medskip
When $|U|$ is finite, the fundamental group scheme $\pi_1(X,U,x)$ is finite. Thus $\pi_1(X,x)$ is the projective limit of finite group schemes, and the universal torsor $\hat X$ is the projective limit of torsors under finite group schemes.

\begin{defi}\label{defTRIPLETS}
We will consider the category whose objects are pointed torsors under finite group schemes, i.e. triples $(Y,G,y)$ where $f:Y\rightarrow X$ is a $G$-torsor under some finite $k$-group scheme $G$ endowed with a section $y:Spec(k)\rightarrow Y$ such that $f(y)=x$.

A morphism $\varphi:(Y_1,G_1,y_1)\rightarrow (Y_2,G_2,y_2)$
between two pointed torsors is the datum of two morphisms
$\alpha:Y_1\rightarrow Y_2$ and $\beta: G_1\rightarrow G_2$ where
$\beta$ is a group scheme morphism, $\alpha(y_1)=y_2$ and such that the
following diagram

\begin{displaymath}
\begin{array}{ccc}
  G_1\times Y_1 & \rightarrow & Y_1 \\
  \downarrow & \circlearrowleft & \downarrow \\
  G_2\times Y_2 & \rightarrow & Y_2
\end{array}
\end{displaymath}
commutes (horizontal arrows being the actions of the concerned group schemes).\\
\end{defi}

\begin{defi}\label{defTRIPLERIDOTTE} A pointed torsor $(Y,G,y)$, as in definition \ref{defTRIPLETS},  is said to be a
quotient torsor if for any pointed torsor $(Y',G',y')$ and any morphism
$\varphi=(\alpha,\beta):(Y',G',y')\rightarrow (Y,G,y)$, $\beta$
is a faithfully flat morphism.\footnote{In \cite{Nor2}, Ch. II. Nori says ``reduced'' instead of ``quotient''. Because of the possible confusion with the usual notion of reduced scheme, we changed the terminology.}\\
\end{defi}

\begin{rem}\label{osservaTRIPLERIDOTTE} A pointed torsor $(Y,G,y)$ is a quotient torsor if and only if $G$ is a quotient of the fundamental group scheme
$\pi_1(X,x)$, that is the canonical morphism $\pi_1(X,x)\to G$ is faithfully flat (cf. for instance \cite{Antei}, Corollary 2.8).\\
\end{rem}

\section{Galois closure of essentially finite morphisms.}

\label{sez:tre}

We keep the notation introduced in Notation \ref{notazCAP4}.

\medskip
\subsection{Statement of the results}

\begin{defi}\label{essentiellement-fini} Let $f:Y\to X$ be a finite and flat morphism (then faithfully flat since finite and flat imply that $f$ is closed and open and as $X$ is connected, $f$ is surjective). The morphism $f$ is said to be essentially finite if and only if $f_{\ast}(\mathcal{O}_Y)$ is an essentially finite vector bundle (cf. def. \ref{defCATEGORIA.GENERATA}). 
\end{defi}
\noindent Let $f:Y\to X$ be an essentially finite morphism. Let us consider the full tannakian subcategory of $EF(X)$ generated by
$f_{\ast}(\mathcal{O}_Y)$ that is
$$EF(X,\{f_{\ast}(\mathcal{O}_Y)\})$$ provided with the fiber
functor $\omega_U:x^{\ast}\circ i_U
:EF(X,U)\to k$-mod, where
$U:=\{f_{\ast}(\mathcal{O}_Y)\}$ and
$i_U:EF(X,U)\hookrightarrow
\mathcal{Q}coh(X)$ is as in section \ref{sez:Intro}. From this data we get the fundamental group scheme $G= \pi _1 (X,U,x)$ and the universal torsor $(\hat X_U , \hat x_U)$. Denote by $p: \hat X_U \to X$ the structural morphism of the universal torsor. The main result of this section is the following theorem.

\begin{thm}\label{universel}
Assume that $\theta: X \to Spec (k)$ is a proper reduced locally noetherian $k$-scheme over a field $k$ such that $H^0 (X,\mathcal{O}_X)=k$, endowed with a rational point $x: Spec (k) \to X$. Let $f: Y \to X$ be an essentially finite morphism such that  $H^0 (X, f_\ast (\mathcal{O}_Y))=k$ given with a rational point $y $ in the fiber $Y_x$ of $f$ at $x$. 
\begin{enumerate}
\item Under these conditions there exists a unique  morphism $\lambda: \hat X_U \to Y$ sending $\hat x _U $ to $y$ and satisfying $ f \circ \lambda = p$. Moreover this morphism is faithfully flat.

\item Moreover $\lambda: (\hat X_U, \hat x _U)\to (Y,y)$ has the structure of a pointed right torsor under the stabiliser $G_y$ of $y$ under the action of $G= \pi _1(X,U,x)$ on the fiber $Y_x$.

\item Finally the universal torsor $(\hat X_U, \hat x_U)$ satisfies the following universal property: for any quotient pointed torsor $(T, H, t)$,  and any faithfully flat morphism $\mu: T \to Y$ such that $\mu (t)=y$, there exists a unique morphism of pointed torsors $(u, \varphi )$ where $\varphi: H \to G= \pi _1 (X,U,x)$ and $u: T\to \hat X_U$ sending $t$ to $\hat x_U$ making the following diagram commutative 

$$\xymatrix{T\ar@/_/[ddr]_{g} \ar@/^/[drr]^{\mu} \ar@{-->}[dr]^{u} & &  \\  & \hat X_U \ar@{->>}[d]_{p}
\ar[r]^{\lambda}& Y\ar@{->>}[dl]^{f}\\ & X &   }$$
\end{enumerate}
\end{thm}
\medskip

In part 1 of theorem \ref{universel} the existence of $\lambda$ follows from Nori's work. The point of the theorem is that $ \lambda $ is fppf.

\begin{nota}
In the situation of theorem \ref{universel}, the group $H$ acts, via the morphism $ \varphi: H \to \pi _1 (X,U,x)$ on the fiber $Y_x$. We will denote by $H_y$ the stabilizer of $y$ in this action.
\end{nota}

\medskip

\begin{defi}\label{domination} Let $f: Y\to X$ be an essentially finite morphism endowed with a rational point $y$ in the fiber of $x$, and $g:T\to X$ a quotient torsor pointed by a rational point $t \in T(k)$. We will say that the pointed torsor dominates the morphism $f$ if there exists a faithfully flat morphism $ \lambda: T \to Y $ such that $g= f \circ \lambda $ and $\lambda (t) = y$.
\end{defi}
\medskip

\begin{cor}\label{correspondance} Let $g:T\to X$ be a quotient torsor under a finite  group scheme $H$ over $k$ pointed by a rational point $t \in T(k)$. 

\begin{enumerate}
\item The correspondence which associates to any essentially finite morphism $f: Y \to X$ pointed by $y \in Y_x(k)$ dominated by $g$, the stabilizer $H_y<H$ is a bijection between pointed essentially finite morphisms $f: Y \to X$ dominated by $g$, up to isomorphism, and closed $k$-subgroup schemes of $H$, up to conjugation. 

\item Moreover, $f: Y \to X$ is a torsor if and only if $H_y$ is normal in $H$; in this case it is a torsor under the group scheme $ H /H_y$.
\end{enumerate}
\end{cor}

\medskip
\begin{rem}
\label{RemEtalEssential}
In the characteristic $0$ case, essentially finite morphisms $f: Y \to X$ such that $H^0(Y,\mathcal{O}_Y)=k$ are just finite \'etale morphisms. Indeed if $f: Y \to X$ is an essentially finite morphism such that $H^0(Y,\mathcal{O}_Y)=k$ then after extension of scalars we may assume that there are points $ x \in X(k)$ and $y \in Y(k)$ such that $f(y)=x$. Theorem \ref{universel} insures the existence of a Galois closure $\hat f:\hat X_U \to X$ which is a torsor under a finite group scheme $G$. As $ch(k)=0$, $G$ is an \'etale group scheme, and then $\hat f:\hat X_U \to X$ is \'etale, which implies that $f: Y \to X$ is itself \'etale.

We will see in section \ref{sez:quattro}  other examples of essentially finite morphisms, namely towers of torsors under finite group schemes.
\end{rem}
\medskip
\begin{rem}
In \cite{Nor3}, Nori shows that the fundamental group scheme of an abelian variety is abelian. It then follows from the Galois correspondence that any essentially finite morphism $f: Y \to X$ satisfying $H^0 ( Y , \mathcal{O}_Y) =k$, where $X$ is an abelian variety defined over a field $k$, is itself a torsor under a finite abelian group scheme.
\end{rem}

Let $f:Y\to X$ be an essentially finite morphism. A natural question that arises is whether the direct image of an essentially finite vector bundle over $Y$ is still essentially finite over $X$. Theorem \ref{teoDEDICATOaWEIL} gives a positive answer in a little more general setting.
 As an application we will show that the composite morphism $f' \circ f$ of  two torsors under finite group schemes $f: Y \to X$ and $f': Y' \to Y$, which in general  is not a torsor itself,  is an essentially finite morphism; therefore we will be able to apply to this morphism the construction of theorem \ref{universel} and to construct a sharp Galois closure for towers of torsor, i.e. the smallest torsor dominating (in the sense of definition \ref{domination}) the morphism $f' \circ f$. We use a recent result of Garuti (cf. \cite{Gar}, Theorem 1) who constructs a ``Galois closure'' of towers of torsors (which is not the smallest possible torsor dominating the tower).

\begin{thm}\label{teoDEDICATOaWEIL} Let $X$  be a reduced, proper $k$-scheme such that $H^0 (X,\mathcal{O}_X)=k$, endowed with a rational point $x \in X(k)$. Let $f: Y \to X$ be an essentially finite morphism of $k$-schemes. Assume either $Y$ to be a reduced, proper $k$-scheme with $H^0 (Y,\mathcal{O}_Y)=k$, endowed with a rational point $y \in Y_x(k)$, or $f: Y \to X$ to be a torsor under a finite group scheme. Let $\mathcal{F}$ be a vector bundle over $Y$ trivialized by a torsor over $Y$ under a finite $k$-group scheme. Then the sheaf
$f_{\ast}(\mathcal{F})$ is an essentially finite vector bundle over $X$.
\end{thm}

The two following corollaries are immediate consequences of this last theorem:

\begin{cor}\label{compos-ess-fin-morph} 
Let $X$  be a reduced, proper $k$-scheme such that $H^0 (X,\mathcal{O}_X)=k$, endowed with a rational point $x \in X(k)$. Let $f: Y\to X$  be a reduced, proper $k$-scheme such that $H^0 (Y,\mathcal{O}_Y)=k$, endowed with a rational point $y \in Y_x(k)$. Let $g:Z\to Y$ be an essentially finite morphism of $k$-schemes. Then $f\circ g$ is essentially finite.\end{cor} 

\proof From Theorem \ref{universel}, there exists some torsor $T \to Y$ under a finite group scheme which trivializes $g_\ast (\mathcal{O}_Z)$. Thus $f_\ast g_\ast (\mathcal{O}_Z)$ is essentially finite. \endproof

\medskip

\begin{cor}\label{tour-de-torseurs}
Let $k$ be a field and $X$ a proper, reduced
$k$-scheme, such that $H^0 (X,\mathcal{O}_X)=k$, provided with a point $x\in X(k)$. Suppose we are given two
finite $k$-group schemes $G$ and $G'$, a $G$-torsor $f:Y\to X$
and a $G'$-torsor $f':Y'\to Y$. Then $(f\circ
f')_{\ast}(\mathcal{O}_{Y'})$ is an essentially finite vector bundle over $X$.
\end{cor}
\proof As $f': Y' \to Y$ is a torsor under $G'$, then $Y' \times _Y Y' \simeq Y'\times G'$. Therefore $f'_\ast (\mathcal{O}_{Y'})$ is trivialized by the torsor $f' : Y' \to Y$. According to Theorem \ref{teoDEDICATOaWEIL}, $(f\circ f')_{\ast}(\mathcal{O}_{Y'})$ is an essentially finite vector bundle. \endproof

From corollary  \ref{tour-de-torseurs} and theorem \ref{universel} we finally obtain the Galois closure of towers
of torsors:

\begin{cor}\label{tour-de-torseurs-2}
Let $k$ be a field and $X$ a reduced proper
$k$-scheme, such that $H^0 (X,\mathcal{O}_X)=k$, provided with a point $x:Spec(k)\to X$.
Let $G$ and $G'$ be two finite $k$-group schemes, $f:Y\to X$ a $G$-torsor
and $f':Y'\to Y$ a $G'$-torsor. We assume the
existence of a point $y:Spec(k)\to  Y$ lying over $x$ and of a
point $y':Spec(k)\to Y'$ lying over $y$. We assume that $H^0 (Y, \mathcal{O}_Y)=H^0 (Y', \mathcal{O}_{Y'})=k$

\begin{enumerate}

\item Then there exists a finite
$k$-group scheme  $\hat {G}$, a $\hat {G}$-torsor
$p: {U}\to X$ pointed by a rational point $u \in U (k)$ above $x$ and a faithfully flat morphism
$\lambda:{U}\twoheadrightarrow Y'$ with a right torsor structure such that $f \circ f' \circ \lambda = p$.

\item Moreover the torsor $p: U \to X$ satisfies the following universal property: for any quotient triple $(T, H, t)$, where $g: T \to X$ is a torsor under a finite $k$-group scheme $H$ and $t$ a $k$-rational point over $x$, and a faithfully flat morphism $\mu: T \to Y'$ such that $\mu (t)=y'$, there exists a unique morphism of pointed torsors $(h, \varphi )$ where $\varphi: H \to \hat G$ and $h: T\to U$ sending $t$ to $u$ making the following diagram commutative 

$$
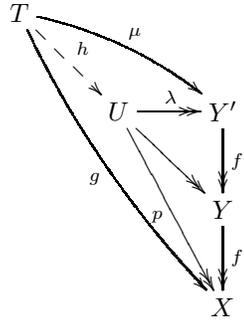
\xymatrix{T\ar@/_/[dddrr]_{g} \ar@/^/[drr]^{\mu} \ar@{-->}[dr]^{h} & &  \\  & U \ar@{->>}[rdd]_{p} \ar@{->>}[rd]
\ar@{->>}[r]^{\lambda}& Y'\ar@{->>}[d]^{f'}\\
 &  &  Y\ar@{->>}[d]^{f}\\
 &&X }$$

\item Denote by $Y'_x$ the fiber of $x$ in the morphism $f \circ f'$, and $Y'_y\subset Y'_x$ the fiber of $y$ in the morphism $f'$. Let $\hat G_{y'}$ (resp. $ \hat G_y$) be the stabilizer of $y'$ (resp. of $Y'_y$) in the action of $\hat G $ on $Y' _x$.

\begin{enumerate}
\item Then $\hat G_y $ is a normal subgroup of $\hat G$; $f' \circ \lambda: \hat X \to Y$ is a right torsor under $\hat G_y$; and $G \simeq \hat G/\hat G_y$ in the Galois correspondence.

\item Also $\hat G_{y'} $ is normal in $\hat G_y$; $\lambda: \hat X \to Y'$ is a torsor under $\hat G _{y'}$; and $G' \simeq \hat G_y / \hat G_{y'}$ in the Galois correspondence.
 \end{enumerate}
 \end{enumerate}
 \end{cor}
\proof According to corollary
\ref{tour-de-torseurs} the sheaf $(f\circ
f')_{\ast}(\mathcal{O}_{Y'})$ is a finite vector bundle. Theorem \ref{universel} insures the existence of a Galois closure $U \to X$ of the essentially finite morphism $f\circ f'$ and it says that $\lambda:U \to Y'$ is a $\hat {G}_{y'}$-torsor where $\hat {G}_{y'}$ is the stabilizer of $y'$ in the action of $\hat G$ on the fiber $Y'_x$. So conclusions 1 and 2 of the theorem are immediate consequences of theorem \ref{universel} applied to the essentially finite morphism $f \circ f'$. To prove 3 first remark that $f' : Y' \to Y$ induces a morphism $Y'_x \to Y_x$ which is  compatible with the actions of $\hat G$. Thus the stabilizer $\hat G _y$ of $Y' _y$ is also the stabilizer of $y$ in the action of $\hat G $ on $Y_x$. Thus the point 3 (a) is a consequence of the point 2 of corollary \ref{correspondance} applied to the torsor $p : U \to X$ and the torsor $f : Y \to X$. Finally by the point 1 of corollary \ref{correspondance},  $UÊ\to Y$ is a torsor under $\hat G_y $ and the point 3 (b) is a consequence again of second point of corollary \ref{correspondance} applied to this torsor and the $G'$-torsor $f' : Y' \to Y$. \endproof

\medskip

\begin{rem}\label{remGarou} We already mentioned in the introduction that the ``Galois closure'' constructed by Garuti in \cite{Gar} is in general bigger than the Galois closure given by theorem \ref{universel}. Here is an example. Let $k$ be a field of characteristic $2$, $A$ be an abelian variety of dimension $2$ over $k$. 
Let us assume that the multiplication by $2$ map $2_A:A'=A\to A$ is a $\mathbf{\alpha}_{2}\times \mathbf{\alpha}_{2}$-torsor. Consider the quotient morphism $\rho:\mathbf{\alpha}_{2}\times \mathbf{\alpha}_{2}\to \mathbf{\alpha}_{2}$ and the contracted product, through $\rho$, $B:=A'\times^{\mathbf{\alpha}_{2}\times \mathbf{\alpha}_{2}} \mathbf{\alpha}_{2}$, then $B\to A$ and $A'\to B$ are both $\mathbf{\alpha}_2$-torsors, thus $A'\to B\to A$ is a tower of torsors. According to Corollary \ref{tour-de-torseurs-2} and theorem \ref{universel} (3) the Galois closure that we construct coincides with the $\mathbf{\alpha}_{2}\times \mathbf{\alpha}_{2}$-torsor  $2_A:A'\to A$. Now we follow the proof of \cite{Gar}, Theorem 2 in order to obtain Garuti's Galois closure : using Garuti's notation consider $\Phi(\mathbf{\alpha}_2,\mathbf{\alpha}_2)$  the $k$-group scheme of morphisms (in the category of $k$-Schemes) from $\mathbf{\alpha}_2$ to $\mathbf{\alpha}_2$ and $\Phi^1(\mathbf{\alpha}_2,\mathbf{\alpha}_2)$ the $k$-subgroup scheme of $\Phi(\mathbf{\alpha}_2,\mathbf{\alpha}_2)$ sending $1_{\mathbf{\alpha}_2}$ to $1_{\mathbf{\alpha}_2}$, then  $\Phi^1(\mathbf{\alpha}_2,\mathbf{\alpha}_2)\simeq \mathbb{G}_a$. Finally, let $F:k\to k$ be the Frobenius morphism,  $\mathbb{G}_a^{(2)}$ the pull back of  $\mathbb{G}_a$ following $F$ and ${}_{F^1}\mathbb{G}_a$ the kernel of $\mathbb{G}_a\to \mathbb{G}_a^{(2)}$ then ${}_{F^1}\mathbb{G}_a\simeq \mathbf{\alpha}_2$ (cf. \cite{DemGab}, II \S 7, 1.5). Garuti's construction thus provides  a torsor $Z\to A$ where  $Z\simeq A'\times \mathbf{\alpha}_2$, which obviously does not coincide with $2_A:A'\to A$.
\end{rem}

\subsection{Preliminary tools}

The following two lemmas are straightforward enough that we leave their proof to the reader:

\begin{lem}\label{lemmaPRODOTTO.CONTRATTO1}Let $\mathcal{C}$ and
$\mathcal{C}'$ be two tannakian categories, $\gamma , \eta:\mathcal{C}\to
\mathcal{Q}coh(S)$ two fiber functors over a $k$-scheme $S$ and $F:\mathcal{C}' \to \mathcal{C}$ an exact 
tensor functor. We have the following relation between torsors
$$\underline{Isom}_S^{\otimes}(\gamma \circ F,\eta\circ F)\simeq\underline{Isom}_S^{\otimes}(\gamma ,\eta)
\times^{\underline{Aut}_S^{\otimes}(\gamma)}\underline{Aut}_S^{\otimes}(\gamma\circ
F)$$ the second term being the contracted product (see for instance \cite{DemGab}, III, \S 4, 3.2).
\end{lem}


\medskip
\begin{lem}\label{2}
Let $\cal C$ be a tannakian category, $j: S' \to S$ a flat morphism of $k$-schemes, $\eta,\gamma$ two fiber functors over $S$, then there is a canonical morphism of right torsors

$$\underline {Isom}_{S'} ^\otimes (j^\ast\circ \eta, j^\ast \circ \gamma) \simeq j^\ast \underline {Isom}_S ^\otimes (\eta,\gamma).$$

\end{lem}
\medskip

\begin{thm}\label{thReferee}
There is a one-to-one correspondence between the following objects:

\begin{enumerate}

\item torsors $T \to X$ under a finite $k$-group scheme $G$ endowed with a rational point $t \in T(k)$ over $x \in X(k)$

\item morphisms $\varphi: \pi _1 (X,x) \to G$

\item exact tensor functors $\gamma: Rep_k (G) \to Qcoh(X)$ satisfying the relation $x ^\ast \circ \gamma \simeq {\rm forget} _{kG}$.
\end{enumerate}
\end{thm}

\begin{proof}
$1\leftrightarrow 3.$ We already know that there is a one-to-one correspondence between (not necessarily pointed) right torsors $f: T \to X$ under the group $G$ over $X$ and tensor functors $\gamma: Rep_k (G) \to Qcoh (X)$ given by the following relation:

$$ T \simeq \underline {Isom}_X ^\otimes (\theta ^\ast \circ {\rm forget}_{kG}, \gamma )$$ where ${\rm forget}_{kG}: Rep_k (G) \to k$-$mod$ is the forgetful functor (cf. theorem \ref{teorDELIGNEMILNE}). The fiber functor $\gamma $ factors through a tensor functor $\tilde \gamma: Rep_k( G) \to EF(X)$, i.e. $\gamma  = i_X \circ \tilde \gamma $, where $i_X$ is the inclusion of $EF(X) $ into the category $Qcoh (X)$ (\cite{Nor2}, Chapter I, Prop. 3.8). The fiber of $x$ in $T$,  $x^\ast T \simeq Isom^\otimes ({\rm forget } _{kG}, x^\ast \circ \gamma)$ has a rational point (i.e. the torsor $T$ is pointed over $x$) if and only if $x^\ast \circ \gamma \simeq {\rm forget } _{kG}$. 

$2\leftrightarrow 3.$ Using the formulas $\gamma \simeq i_X \circ \tilde \gamma $ and $x^\ast \circ i_X \simeq {\rm forget} _{k\pi _1 (X,x)} \circ \tilde x$, the relation $x ^\ast \circ \gamma \simeq {\rm forget} _{kG}$ can be rewritten in the following form :

$${\rm forget}_{k \pi _1 (X,x)} \circ \tilde x \circ  \tilde \gamma \simeq {\rm forget} _{kG}$$
Thus the tensor functor $\tilde x \circ \tilde \gamma : Rep _k (G) \to Rep _k ( \pi _1 (X,x))$ is equivalent to the data of a morphism $\varphi: \pi _1 (X,x) \to G$. 
\end{proof}

If we are given a morphism $\varphi: \pi _1 (X,x) \to G$, we consider the contracted product 

$$\hat{X}\times ^{\pi _1 (X,x)} G$$  with respect to the morphism $\varphi$ which is a right $G$-torsor. The following proposition explains that this operation gives the correspondence $1\leftrightarrow 2$

\begin{prop}\label{produit-contracte}
Suppose that the $G$-torsor $T$ corresponding to the tensor functor $\gamma : Rep _k(G) \to Qcoh (X)$ is pointed over $x$, then $T \simeq \hat{X}\times ^{\pi _1 (X,x)} G$.

\end{prop}

\proof Recall that $\hat{X} = \underline {Isom}_X ^\otimes (\theta ^\ast \circ \omega , i_X)$ where $\omega=x^\ast \circ i_X$. Using lemma \ref{lemmaPRODOTTO.CONTRATTO1}, one has 
$$(\dagger ) \quad \hat{X} \times ^{\pi _1(X,x)} G \simeq \underline {Isom}_X^\otimes  (\theta ^\ast \circ \omega \circ \tilde \gamma , i_X\circ \tilde \gamma)\simeq \underline {Isom}_X^\otimes  (\theta ^\ast \circ x^\ast \circ \gamma , \gamma) $$  Using lemma \ref{2} and the definition of $T$, one gets

$$x^\ast T \simeq \underline {Isom}_k ^\otimes (x^\ast \circ \theta ^\ast \circ {\rm forget}_{kG}, x^\ast \circ \gamma )= \underline {Isom}_k ^\otimes ({\rm forget}_{kG}, x^\ast \circ \gamma ) $$

The fact that $T$ has a $k$-point over $x$ means that $x^\ast T$ is trivial, and then the functors ${\rm forget}_{kG}$ and $ x^\ast \circ \gamma $ are equivalent. 

Replacing in the formula $(\dagger )$, we get

$$ \quad \hat{X} \times ^{\pi _1(X,x)} G \simeq \underline {Isom}_X^\otimes  (\theta ^\ast \circ {\rm forget}_{kG} , \gamma)\simeq T$$ which completes the proof of the proposition. \endproof

\medskip

\begin{prop}\label{representation-reguliere}
Under the hypothesis of proposition \ref{produit-contracte}, there is an isomorphism of $\mathcal{O}_X$-algebras

$$f_\ast (\mathcal{O}_T) \simeq \tilde \gamma (kG)$$
where $kG$ denotes the regular representation of $G$ ($G= Spec (kG)$).
\end{prop}

\proof Recall the following commutative diagrams of functors:

$$\xymatrix{
EF(X) \ar[r]^{x ^\ast } \ar[dr]_ {\tilde x}& k\text{-}mod \\
& Rep_k ( \pi _1 (X,x)) \ar[u]_{\rm forget_{k\pi _1 (X,x)}}
}$$
$$\xymatrix{
Rep_k(G) \ar[r]^{\gamma  } \ar[dr]_ {\tilde \gamma}& Qcoh (X) \\
& EF(X) \ar[u]_{i_X}
}$$

So $x^\ast T= \underline{Isom}_k^\otimes (\rm forget _{kG} , forget_{k \pi _1 (X,x)} \circ \tilde x \circ \tilde \gamma ) \simeq G$ viewed with the left action of $\pi _1 ( X,x)$ on $G$ defined by the morphism $\varphi: \pi _1 (X,x) \to G$ induced by the functor $\tilde x \circ  \tilde \gamma: Rep_k (G) \to Rep_k ( \pi _1 (X,x))$. As $\tilde x$ is an equivalence of categories, $f_\ast (\mathcal{O}_T) = (\tilde x )^{-1} (V)$, where $V$ is the regular representation $kG$ viewed as a representation of $\pi _1 (X,x)$ through $\varphi $, i.e. $V= \tilde x  \circ \tilde \gamma (kG)$. One concludes that $f_\ast (\mathcal{O}_T) \simeq (\tilde x )^{-1} \circ \tilde x  \circ \tilde \gamma (kG)\simeq  \gamma (kG)$. As the functors involved in the proof are tensor functors, they make correspond $k$-algebras and $\mathcal{O}_X$-algebras, and the isomorphism $f_\ast (\mathcal{O}_T)\simeq  \tilde \gamma (kG)$ is thus an isomorphism of $\mathcal{O}_X$-algebras. \endproof
\medskip

\begin{prop} Under the hypothesis of proposition \ref{produit-contracte} the essential image of $\tilde \gamma $ is constituted by the objects of $EF(X)$ trivialized by the torsor $f: T \to X$.
\end{prop}

\proof In one direction it is obvious: as $T = \underline{Isom}_X ^\otimes ( \theta ^\ast \circ {\rm forget}_{kG} , i_X \circ \tilde \gamma )$ and $f ^\ast T=\underline{Isom}_X ^\otimes (f^\ast \circ \theta ^\ast \circ {\rm forget}_{kG} , f^\ast \circ  i_X \circ \tilde \gamma )$ is trivial, for any representation $V$ of $G$, $f^\ast \circ  i_X \circ \tilde \gamma (V)$ is isomorphic to $ f^\ast \circ \theta ^\ast \circ {\rm forget} _{kG} (V)$ which is a trivial vector bundle. 

In the other direction, if $F$ is an essentially finite vector bundle on $X$ which is trivialized by $f: T \to X$, then $f^\ast F \simeq \mathcal{O}_T \oplus \dots \oplus \mathcal{O}_T$, and then $$f_\ast f^\ast F \simeq f_\ast (\mathcal{O}_T) \oplus \dots \oplus f_\ast (\mathcal{O}_T).$$ Moreover, by proposition \ref{representation-reguliere}, $f_\ast (\mathcal{O}_T)$ is the image by $\tilde \gamma $ of the regular representation of $G$ and then $f_\ast (\mathcal{O}_T) \oplus \dots \oplus f_\ast (\mathcal{O}_T)$ is in the essential image of $\tilde \gamma $. As $f$ is faithfully flat, $F \hookrightarrow f_\ast f^\ast F$, and then $F$ is a sub-object of an object of the tannakian category generated by $f_\ast (\mathcal{O}_T)$ and thus is an object of the essential image of $\tilde \gamma $. \endproof


\begin{prop}\label{torseurs-quotients}
With the previous notation, the following statements are equivalent 

\begin{enumerate}
\item $H^0 ( T, \mathcal{O}_T) =k$

\item $\varphi $ is faithfully flat

\item $\gamma $ is fully faithful
\end{enumerate}
\end{prop}

The proof, for which we refer the reader to \cite{Nor2}, Chapter II, Proposition 3, relies on the following lemma, that will be used later:

\begin{lem}\label{invariants}
Let $G$ be an affine group scheme and $\Phi: Rep_k (G) \to EF(X)$ a fully faithful tensor functor. Then for any representation $V$ of $G$, $H^0 (X, \Phi (V))\simeq V^G$.
\end{lem}

\proof We have the following equalities:

$$V^G \simeq Hom_G (V ^{\rm v}, k)\simeq Hom ( \Phi (V)^{\rm v} , \mathcal{O}_X)\simeq $$ $$\simeq H^0 ( X , \underline {Hom} (\Phi (V)^{\rm v} , \mathcal{O}_X ))  \simeq H^0 (X, \Phi (V)).$$ \endproof 
\medskip

\begin{cor} Let $f: T \to X$ be a pointed torsor under a finite group scheme $G$. Then it is a quotient torsor if and only if $H^0 (T,\mathcal{O}_T)=k$.
\end{cor}
\medskip

\medskip

\begin{cor}\label{cas-des-torseurs} Let $f: T\to X$ be a $G$-torsor pointed on $t \in T_x (k)$, where $G$ is a finite $k$-group scheme, and $\varphi: \pi _1 (X,x) \to G$ the corresponding morphism. Consider the tannakian category $EF(X, \{ f_\ast ({\cal O} _T) \} )$. Then the fundamental group scheme $\pi _1 (X, \{ f_\ast ({\cal O} _T) \} , x)$ is isomorphic to the image $H$ of $\varphi $. It is a closed subgroup of $G$, and is equal to $G$ if and only if the pointed torsor $(T,t)$ is a quotient torsor, and in this case the universal torsor of the category $EF(X, \{ f_\ast ({\cal O} _T) \} )$ based at $x$ is isomorphic to $f:T\to X$.
\end{cor}

\proof  Examine the proof of the preceding proposition. \endproof

\medskip
\subsection{Proof of Theorem \ref{universel}}

We keep the notation introduced in the statement of theorem \ref{universel}. We will need the following lemma.

\begin{lem}\label{action-transitive}
If $H^0 (Y, \mathcal{O}_Y)\simeq k$, the morphism $\rho:G\to Y_x$ defined by $g \to g\cdot y$ is
faithfully flat and induces an isomorphism $G/G_y \simeq Y_x$.
\end{lem}

\proof First of all we
observe that for any $k$-linear representation $(V,\sigma)$ of $G=Spec(A)$ the set of fixed
elements i.e. the set of
all those $v\in V$ such that $g\cdot (v\otimes 1_R)=(v\otimes
1_R)$ (for any $k$-algebra $R$ and any $g\in G(R)$) coincides with
the set of all $v\in V$ such that $\gamma(v)=v\otimes 1_A$ where
$\gamma:V\to V\otimes A$ is the comodule structure associated to
$(V,\sigma)$. 
Now set $Y_x:=Spec(V)$ where $V$ is a $k$-linear
representation of $G$ with an additional $k$-algebra structure.
The quotient $Y_x/G$ is the cokernel of the double arrow
$$\xymatrix@1{
 G\times Y_x \ar@<0.5ex>[r]^-{q_1} \ar@<-0.5ex>[r]_-{q_2} & Y_x}
$$ where $q_1:G\times Y_x\to Y_x$ maps $(g,z)\mapsto z$ and $q_2:G\times Y_x\to
Y_x$ (the action) maps $(g,z)\mapsto g\cdot z$ in the category of sheaves for the fppf topology. It is represented by $V^G$, the kernel of the double arrow
$$\xymatrix@1{
 V \ar@<0.5ex>[r]^-{u} \ar@<-0.5ex>[r]_-{\gamma} & V\otimes A}
$$ where $\gamma:V\to V\otimes A$ is the coaction and $u:V\to V\otimes A$ maps $v\mapsto v\otimes 1$ (cf. the proof of the affine case in Th. 3.2, III, \S 2, 4.1 to 4.4 of \cite{DemGab}). According to lemma \ref{invariants}, $V^G\simeq k$:  indeed apply the lemma to the equivalence $F:Rep_k(G)\to EF(X,U)$ associated to the universal torsor
$p:\hat {X}_U\to X$; $F(V)\simeq
f_{\ast}(\mathcal{O}_Y)$, so in particular
$H^0 (X, f_{\ast}(\mathcal{O}_Y))\simeq H^0 (Y,\mathcal{O}_Y) \simeq k$ by
assumption. So $Y_x /G \simeq Spec (k)$. It follows from this that $G \times  Y_x \to Y_x \times Y_x$ is surjective for the fppf topology (Th. 3.2, b, III, \S 2 of \cite{DemGab}) and thus the morphism $ \rho : G \to Y_x$ is surjective for the fppf topology.

On the other hand, $\rho $ induces a monomorphism of fppf sheaves $G/G_y \to Y_x$, where $G_y$ denotes the stabilizer of $y$ (\cite{DemGab}, III, \S 3, 1.6). And then $G/G_y \simeq Y_x$. 

Finally, according to \cite{DemGab}, III, \S 3, 2.5, the morphism $ G \to G/G_y$ is faithfully flat. This concludes the proof of lemma. \endproof

\medskip
{\it Proof of Theorem \ref{universel}}. 

\begin{itemize}

\item Nori defines the functor $F:Rep_k(G)\to EF(X,U)$, already mentioned in the proof of lemma \ref{action-transitive}, which is an equivalence between the category of representations of $G=\pi _1 (X,U,x)$ and $EF(X,U)$. 
As described in \cite{Del}, 7.5-7.12 the functor $F$ induces an equivalence between finite $G$-schemes over $k$ and finite morphisms $g:Z\to X$ such that $g_{\ast}(\mathcal{O}_Z)\in EF(X,U)$. In this equivalence $G$ corresponds to the universal torsor $\hat X _U$ and $Y_x$ corresponds to $Y$. Then the $k$-morphism $\rho $ corresponds to an $X$-morphism $\lambda: \hat X_U \to Y$:

$$
\xymatrix{
\hat X_U \ar[r]^\lambda  \ar@{->>}[d] & Y\ar@{->>}[dl] \\
X& \\}
$$
whose fiber at $x: Spec (k) \to X$ is precisely $\rho $. At this step it has not been proved yet that the morphism $\lambda$ is faithfully flat.

\item The facts that $G/G_y\simeq Y_x$ and that the action $m$ of $G_y$ on $G$ is free imply that the morphism $$(1) \quad \xymatrix {G\times G_y \ar[r]^ { pr_1\times m } & G\times_{Y_x}G }$$ is an isomorphism. This is an isomorphism of $G$-schemes, where $G$ is endowed with the left action of $G$ on itself and $G_y$ with the trivial action of $G$. The image of this diagram by the equivalence of tannakian categories $F: Rep _k(G) \to EF(X,U)$ is the following isomorphism

$$(2) \quad \hat {X}_U\times G_y\to
\hat {X}_U\times_Y \hat {X}_U$$

whose fiber at $x$ is given by the isomorphism $(1)$.

To prove both point 1 and 2 of the statement, the only thing to check is that $\lambda: \hat X_U \to Y$ is faithfully flat. Consider again the commutative diagram 

$$
\xymatrix{
\hat X_U \ar[r]^\lambda  \ar@{->>}[d]_p & Y\ar@{->>}[dl] \\
X& \\}
$$
and pull it back by $\hat X_U \to X$. As we have seen in the proof of proposition \ref{produit-contracte}, the functors $p^\ast $ and $p^\ast \circ \theta ^\ast \circ x^\ast $ from $EF(X,U) $ to $Qcoh (\hat X_U)$ are equivalent. Applying the two functors to the preceding diagram one gets that the pull back by $p$ is 
$$
\xymatrix{
\hat X_U\times  G \ar[r]^{1_{\hat X_U}\times \rho }  \ar@{->>}[d]_{pr_1} & \hat X_U \times  Y_x\ar@{->>}[dl] \\
\hat X_U& \\}
$$
As $ p$ and $\rho$ are faithfully flat, $\lambda $ is also faithfully flat.

 \item In order to prove point 3 of the statement we argue as follows: since $f$ is affine, $f_{\ast}$ is exact; so from the inclusion  $\mathcal{O}_Y\hookrightarrow
\mu_{\ast}(\mathcal{O}_T)$ we get the inclusion
 $f_{\ast}(\mathcal{O}_Y)\hookrightarrow
f_{\ast}(\mu_{\ast}(\mathcal{O}_T))\simeq
g_{\ast}(\mathcal{O}_T)$. Being semi-stable, $f_{\ast}(\mathcal{O}_Y)$ is a sub-object of $g_{\ast}(\mathcal{O}_T)$ and then an object of the tannakian category $EF(X,
\{ g_{\ast}(\mathcal{O}_T)\} ) $. Thus the inclusion is a fully faithful
functor of tannakian categories $$EF(X,
\{ f_{\ast}(\mathcal{O}_Y)\} )\hookrightarrow EF(X,
\{ g_{\ast}(\mathcal{O}_T)\} )$$ which induces a faithfully flat
morphism

$$H\twoheadrightarrow \pi _1 (X,\{ f_{\ast }({\mathcal O } _Y) \} , x)$$ 

between their
tannakian Galois group schemes. As $T$ is the universal torsor associated to the tannakian category  $EF(X, \{g_{\ast}(\mathcal{O}_T)\} )$, from this morphism one gets  $u:T\twoheadrightarrow \hat X_U$ commuting with the actions of $H$ and 
 $\pi _1 (X,\{ f_{\ast }({\mathcal O } _Y) \} , x)$. The same kind of arguments used in the second part of the proof shows that this morphism is also faithfully flat.
\medskip

{\it Proof of Corollary \ref{correspondance}}. 

If $f: Y \to X$ is an essentially finite morphism pointed at $ y \in Y_x (k)$, let $U=\{ f_{\ast }({\mathcal O } _Y) \} $ and $G = \pi _1 (X,U,x)$. 

Suppose that $f: Y \to X$ is dominated by the pointed torsor $g: T \to X$: there exists a faithfully flat morphism $\mu: T \to Y$ such that $f\circ \mu = g$ and $\mu (t)=y$. By Theorem \ref{universel}, there exists a unique morphism of torsors $(u, \varphi )$, where $ \varphi: H \to G$ is a morphism of groups and $u: T \to \hat X_U$ is a morphism of torsors making the following diagram commutative

$$\xymatrix{T\ar@/_/[ddr]_{p} \ar@/^/[drr]^{\mu} \ar@{-->}[dr]^{u} & &  \\  & \hat X_U \ar@{->>}[d]_{p}
\ar[r]^{\lambda}& Y\ar@{->>}[dl]^{f}\\ & X &   }$$

To these morphisms of torsors correspond morphisms of fundamental groups 

$$\xymatrix{\pi _1(X, x) \ar@{->>}[r] &H \ar@{->>}[r] &G\\}$$

One gets an action of $H$ on $Y_x$ and the stabilizer $H_y$ of $y$ under this action. 

Conversely if $H'<H$ is a subgroup of $H$,  the quotient $H/H'$ is endowed with an action of the fundamental group scheme $\pi _1 (X,x)$ through the morphism $\pi _1 (X,x) \to H$ attached to the torsor $p: T \to X$. To the $\pi _1 (X,x)$-$k$-scheme $H/H'$ corresponds a $X$-scheme $f: Y \to X$ such that $f_\ast ( \mathcal{O} _X)$ is essentially finite and $ Y_x \simeq H/H'$. Moreover $Y$ is pointed at $y \in Y_x (k)$ corresponding to the image of the unit element in $H/H'$ and $H'=H_y$.

These correspondences are inverse of each other as in the situation considered above $Y_x \simeq G/G_y \simeq H/H_y$ as $\pi _1 (X,x)$-$k$-schemes.

If the morphism $f: Y \to X$ is a torsor, then it is a quotient torsor under the group $G$. Then $G_y= 1$ and $H_y$ which is the inverse image of $G_y$ in the morphism $ H \to G$ is a normal subgroup of $H$. In this case $G \simeq H/H_y$ and thus $ f: Y \to X$ is a torsor under the quotient group $H/H_y$. 

Conversely if $H_y$ is normal in $H$, $G_y$ is normal in $G$, and as $Y_x \simeq G/G_y$ as representations of $ \pi _1 (X,x)$, the fundamental group $G = \pi _1 (X,U,x)$ which is the image of $\pi _1 ( X,x)$ in this representation is isomorphic to the group $G/G_y$. Thus $G_y=1$ and $f: Y \to X$ is a torsor under $G$.

\end{itemize}

\medskip

\begin{rem} In the situation of theorem \ref{universel} (with the only difference that we do not need the assumption $H^0 (X, f_\ast (\mathcal{O}_Y))=k$), $Y_x=Spec(V)$ where $V$ is naturally a representation of $\pi _1 (X,x)$ which factors through the morphism $\varphi: \pi _1 (X,x) \to G$ associated to the universal torsor $ (\hat X _U , \hat x_U)$. Then $Ker ( \varphi )$ is the kernel of the representation of $\pi _1 (X,x)$ on $V$.  

Let indeed $K$ be this kernel. The inclusion $ Ker(\varphi ) \subset K$ is obvious. In the other direction, $V$ is a representation of $\pi _1 (X,x) /K$, and as $V$ generates the tannakian category $Rep_k (G)$, one has the inclusion $$Rep_k (G) \hookrightarrow Rep_k ( \pi _1 (X,x) /K)$$ which induces a surjective morphism $\pi _1 (X,x) /K \to \pi _1 (X,x) /Ker ( \varphi ) =G$. Thus $K=Ker( \varphi )$ and $G \simeq \pi _1 (X,x) /K$.
\end{rem}

\subsection{Proof of Theorem \ref{teoDEDICATOaWEIL}}
\label{sez:proof}
Throughout this section $k$ will be any field. We first recall that for an integral and projective curve $C$ over $k$ and any coherent sheaf $\mathcal{F}$ over $C$ we define, respectively, the rank and the degree of $\mathcal{F}$ as follows: $$rk(\mathcal{F}):=dim_{k(\xi)}(\mathcal{F}_{\xi})\qquad deg(\mathcal{F}):=\chi(\mathcal{F})-rk(\mathcal{F})\cdot\chi(\mathcal{O}_C)$$ where $\xi$ is the generic point of $C$ and $\chi(\mathcal{F})$ is the Euler-Poincar\'e characteristic of $\mathcal{F}$. Assume moreover that $C$ is normal and let $0\rightarrow \mathcal{F}'\rightarrow \mathcal{F}\rightarrow \mathcal{F}''\rightarrow 0$ be an exact sequence of coherent sheaves over $C$ then clearly  $rk(\mathcal{F})=rk(\mathcal{F}')+rk(\mathcal{F}'')$ and consequently
$deg(\mathcal{F})=deg(\mathcal{F}')+deg(\mathcal{F}'')$. If $\mathcal{V}$ and $\mathcal{W}$ are locally free sheaves over an integral curve $C$ over $k$ then one can compute the degree of $\mathcal{V}\otimes_{\mathcal{O}_C} \mathcal{W}$ as follows (cf. \cite{TS}, Ch 6, \S 7, ex. 9):

$$deg(\mathcal{V}\otimes_{\mathcal{O}_C} \mathcal{W})=rk(\mathcal{V})deg(\mathcal{W})+rk(\mathcal{W})deg(\mathcal{V}).$$
If $\mathcal{F}$ is a coherent $\mathcal{O}_C$-sheaf then the formula still holds:

$$(\star ) \quad deg(\mathcal{F}\otimes_{\mathcal{O}_C} \mathcal{W})=rk(\mathcal{F})deg(\mathcal{W})+rk(\mathcal{W})deg(\mathcal{F}).$$

\noindent Indeed $\mathcal{F}$ has a locally free resolution of length $1$ (cf. \cite{Har}, III, Example 6.5.1, Proposition 6.11 A and Exercise 6.5). This implies that there exist two locally free $\mathcal{O}_C$-modules $\mathcal{L}_0$ and $\mathcal{L}_1$ and an exact sequence  $0\to \mathcal{L}_1\to \mathcal{L}_0\to \mathcal{F}\to 0$ that we tensor by $\mathcal{W}$ thus obtaining another exact sequence 
$$0\to \mathcal{L}_1\otimes_{\mathcal{O}_C} \mathcal{W}\to \mathcal{L}_0\otimes_{\mathcal{O}_C} \mathcal{W}\to \mathcal{F}\otimes_{\mathcal{O}_C} \mathcal{W}\to 0,$$ whence the formula,  by the additivity of the degree. We need the following lemma, whose easy proof is left to the reader (otherwise compare with \cite{HuLe}, Lemma 3.2.1):

\begin{lem}\label{lemme413}
Let $C$ and $C'$ be integral and projective curves over $k$ and $f:C'\to C$ a finite 
morphism of degree $d$. Let $\mathcal{F}$ be a coherent
$\mathcal{O}_{C'}$-module and $\mathcal{G}$ a coherent
$\mathcal{O}_{C}$-module
 then  $deg(f_{\ast}(\mathcal{F}))=deg(\mathcal{F})+deg(f_{\ast}(\mathcal{O}_{C'}))\cdot rk(\mathcal{F})$; if moreover $f$ is flat or $\mathcal G$ locally free then $deg(f^*(\mathcal{G}))=d\cdot deg(\mathcal{G})$. \end{lem} 
 

\begin{lem}\label{lemmeEspoir}
Let  $C$ be a normal and proper curve over $k$, $C'$ a curve over $k$ and $f':C'\to C$ a finite and flat morphism. Denote by $d=rk_{\mathcal{O}_{C}}(f'_{\ast}(\mathcal{O}_{C'}))$ the degree of $f'$ . Let $\mathcal{A}$ and $\mathcal{B}$ be, respectively, a locally free and a coherent sheaf over $C'$. Let us denote the coherent $\mathcal{O}_C$-module $f'_{\ast}(\mathcal{A}\otimes_{\mathcal{O}_{C'}}\mathcal{B})$ by $\mathcal{F}$. Then $$deg(\mathcal{F})= d^{-1}\biggl(rk(f'_{\ast}(\mathcal{A}))deg(f'_{\ast}(\mathcal{B}))+rk(f'_{\ast}(\mathcal{B}))deg(f'_{\ast}(\mathcal{A}))-deg(f'_{\ast}(\mathcal{O}_{C'}))rk(\mathcal{\mathcal{F}})\biggr)$$\end{lem}

\proof  We let $\mathcal{M}$ denote $\mathcal{F}\otimes_{\mathcal{O}_{C}}f'_{\ast}(\mathcal{O}_{C'})$ and we compute its degree: $$deg(\mathcal{M})=deg(\mathcal{F})rk(f'_{\ast}(\mathcal{O}_{C'}))+deg(f'_{\ast}(\mathcal{O}_{C'}))rk(\mathcal{F}).$$
We will prove below the formula $\mathcal{M}\simeq  f'_{\ast}(\mathcal{A})
\otimes_{\mathcal{O}_{C}} f'_{\ast}(\mathcal{B})$, thus, as $f'$ is flat, $ f'_{\ast}(\mathcal{A})$ is locally free and the formula $ (\star ) $ gives  $$deg(\mathcal{M})=rk(f'_{\ast}(\mathcal{B}))deg(f'_{\ast}(\mathcal{A}))+rk(f'_{\ast}(\mathcal{A}))deg(f'_{\ast}(\mathcal{B}))$$ which is enough to conclude. It only remains to prove the isomorphism $\mathcal{M}\simeq  f'_{\ast}(\mathcal{A})
\otimes_{\mathcal{O}_{C}} f'_{\ast}(\mathcal{B})$ as $\mathcal{O}_C$-sheaves:  we have the following isomorphisms
$$f'_\ast (\mathcal{B}) \otimes_{\mathcal{O}_C} f'_\ast (\mathcal{A})\simeq f'_\ast (\mathcal{B} \otimes_{\mathcal{O}_{C'}}f'^\ast f'_\ast (\mathcal{A}))\simeq f'_\ast (\mathcal{B} \otimes_{\mathcal{O}_{C'}}\mathcal{A}\otimes_{\mathcal{O}_{C'}}f'^\ast f'_\ast (\mathcal{O}_{C'}))\simeq$$ $$\simeq f'_\ast (\mathcal{B} \otimes_{\mathcal{O}_{C'}}\mathcal{A})\otimes_{\mathcal{O}_{C}} f'_\ast (\mathcal{O}_{C'})$$ 
where the first and third isomorphisms hold by the projection formula (cf. \cite{Har}, II, Exercise 5.1) and the second is a consequence of the isomorphism $$ f'^*f'_* \mathcal{G}\simeq \mathcal{G}\otimes_{\mathcal{O}_{C'}}f'^*f'_*(\mathcal{O}_{C'}),$$ which holds for any coherent $\mathcal{O}_{C'}$-module $\mathcal{G}$ since $f'$ is affine.
\endproof

\begin{defi} Let $\mathcal{F}$ be a vector bundle over a $k$-scheme $T$ such that\\ $deg(i^{\ast}(\mathcal{F}))=0$ for any proper and normal $k$-curve $D$ and any non constant morphism $i:D\to T$: we will say that $\mathcal{F}$ has restricted degree $0$.\end{defi}

\begin{lem}\label{lemmaSperanza} Let  $X$ and $Y$ be two $k$-schemes and let $f:Y\to X$ be a finite and flat morphism such that $f_{\ast}(\mathcal{O}_Y)$ has  restricted degree $0$. Let $\mathcal{F}$ be a vector bundle over $Y$ which has restricted degree $0$ . Then the vector bundle $f_{\ast}(\mathcal{F})$ has restricted degree $0$.
\end{lem}
\proof
We denote $rk(f_{\ast}(\mathcal{O}_Y))$ by $d$. Let $C$ be a normal and proper curve over $k$, $j:C\to X$ a non constant morphism and $C':=Y\times_X C$; we consider the following diagram:
$$\xymatrix{\widetilde{C}\ar[rd]^{s} & & \\ & C'\ar[r]^{j'}\ar[d]_{f'}& Y\ar[d]^{f} \\  & C\ar[r]_{j} & X}$$ where $\widetilde{C}$ is the normalization of an irreducible component (surjective over $C$) of the curve $C'_{red}$ obtained by $C'$ after reduction. We want to show that $deg(j^{\ast}f_{\ast}(\mathcal{F}))=0$. We know that
\begin{enumerate}
	\item $deg(s^{\ast}j'^{\ast}(\mathcal{F}))=0$ by assumption since $j' \circ s$ is not constant because $f'\circ s$ is surjective;
	\item $f'_{\ast}j'^{\ast}(\mathcal{F})\simeq j^{\ast}f_{\ast}(\mathcal{F})$;
	\item for any quasi coherent $\mathcal{O}_{C'}$-module $\mathcal{G}$ we have $s_{\ast}s^{\ast}(\mathcal{G})\simeq \mathcal{G}\otimes_{\mathcal{O}_{C'}}s_{\ast}(\mathcal{O}_{\widetilde{C}})$.
\end{enumerate}

\noindent We denote $f'_{\ast} s_{\ast}s^{\ast}j'^{\ast}(\mathcal{F})$ by $\mathcal{P}$. According to lemma \ref{lemme413} and point 1, we know that  $$deg(\mathcal{P})=deg(f'_{\ast} s_{\ast}(\mathcal{O}_{\widetilde{C}}))\cdot rk(\mathcal{F}).$$ Using point 3, we obtain $$\mathcal{P}=f'_{\ast} s_{\ast}s^{\ast}j'^{\ast}(\mathcal{F})\simeq 
f'_{\ast}(j'^{\ast}(\mathcal{F})\otimes_{\mathcal{O}_{C'}} s_{\ast}(\mathcal{O}_{\widetilde{C}})).$$ 
We observe that $deg(f'_{\ast}(\mathcal{O}_{C'}))=0$: indeed by point 2 we have $f'_{\ast}(\mathcal{O}_{C'})\simeq f'_{\ast}j'^{\ast}(\mathcal{O}_Y)\simeq j^{\ast}f_{\ast}(\mathcal{O}_Y)$  then $deg(f'_{\ast}(\mathcal{O}_{C'}))=0$ by assumption.
Hence, according to lemma \ref{lemmeEspoir}  $deg(\mathcal{P})=$ $$=rk(f'_{\ast}(\mathcal{O}_{C'}))^{-1}\biggl(rk(f'_{\ast}s_{\ast}(\mathcal{O}_{\widetilde{C}}))deg(f'_{\ast}j'^{\ast}(\mathcal{F}))+rk(f'_{\ast}j'^{\ast}(\mathcal{F}))deg(f'_{\ast}s_{\ast}(\mathcal{O}_{\widetilde{C}}))\biggr)=$$ 
$$=d^{-1}\biggl(rk(f'_{\ast}s_{\ast}(\mathcal{O}_{\widetilde{C}}))deg(f'_{\ast}j'^{\ast}(\mathcal{F}))\biggr)+d^{-1}\biggl( d\cdot rk(\mathcal{F}) \cdot deg(f'_{\ast}s_{\ast}(\mathcal{O}_{\widetilde{C}}))\biggr);$$
but we already know that $deg(\mathcal{P})=deg(f'_{\ast} s_{\ast}(\mathcal{O}_{\widetilde{C}}))\cdot rk(\mathcal{F})$ thus $deg(f'_{\ast}j'^{\ast}(\mathcal{F}))=0$ and by point 2 we finally obtain $deg(j^{\ast}f_{\ast}(\mathcal{F}))=0$.
\endproof

\medskip
{\it Proof of Theorem \ref{teoDEDICATOaWEIL}}.
We are given a proper reduced $k$-scheme $X$ such that $H^0(X,\mathcal{O}_X)=k$ and an essentially finite morphism $f:Y\to X$ endowed with a section $y\in Y_x(k)$. Let $g:Z\to Y$ be a torsor under a finite $k$-group scheme $G_2$ trivializing $\mathcal{F}$ (i.e. $g^{\ast}(\mathcal{F})\simeq \mathcal{O}_Z^{\oplus r}$ where $r=rk(\mathcal{F})$). First we need to observe that $\mathcal{F}$ has restricted degree $0$: so let $C$ be an integral, proper and normal $k$-curve, $j:C\to Y$ a non constant morphism and $C':=C\times_Y Z$. Let us consider the following diagram:
$$\xymatrix{\widetilde{C}\ar[rd]^{s} & & \\ & C'\ar[r]^{j'}\ar[d]^{g'}& Z\ar[d]^{g} \\  & C\ar[r]_{j} & Y}$$

\noindent where $\widetilde{C}$ is an irreducible component of the curve $C'_{red}$ obtained by $C'$ after reduction then $$deg(s^*g'^*j^*\mathcal{F})=deg(s^*j'^*g^*\mathcal{F})=deg(\mathcal{O}_{\widetilde{C}}^{\oplus r})=0$$ and by means of lemma \ref{lemme413} $deg(j^*(\mathcal{F}))=0$. If $f: Y\to X$ is a torsor under a finite group scheme, take $T_1=Y$. If not by theorem \ref{universel} there  exists a torsor $T_1$ over $X$ under a finite $k$-group scheme $G_1$ dominating $f$. Set $T_2:=T_1\times_Y Z$. It is a $G_2$-torsor over $T_1$. According to \cite{Gar}, Theorem 1 there exist two torsors under finite $k$-group schemes $u:T\to X$ and $r:T\to T_2$ such that the diagram   $$\xymatrix{T \ar[r]^{r}\ar@/_/[rdd]_{u} & T_2 \ar[r]^t \ar[d]& Z \ar[d]^{g}\\ & T_1 \ar[r] \ar[d] & Y \ar[ld]^{f}\\  & X &  }$$ commutes. Set $h:=t\circ r:T\to Z$; it is a faithfully flat morphism. Then $h^{\ast}g^{\ast}(\mathcal{F})\simeq \mathcal{O}_T^{\oplus r}$ thus $u_{\ast}h^{\ast}g^{\ast}(\mathcal{F})$ is a finite vector bundle because $u_{\ast}(\mathcal{O}_{T})$ is. Since $g\circ h$ is faithfully flat then $\mathcal{O}_Y\hookrightarrow (g\circ h)_{\ast}(\mathcal{O}_T)$ which implies 
$$\mathcal{F}\simeq \mathcal{F}\otimes_{\mathcal{O}_Y}\mathcal{O}_Y \hookrightarrow  \mathcal{F}\otimes_{\mathcal{O}_Y} (g\circ h)_{\ast}(\mathcal{O}_T)\simeq (g\circ h)_{\ast}(g\circ h)^{\ast}(\mathcal{F}).$$ This means  that $\mathcal{F}$ is a subbundle of $(g\circ h)_{\ast}(g\circ h)^{\ast}(\mathcal{F})$: indeed  the inclusion $\mathcal{O}_Y\hookrightarrow (g\circ h)_{\ast}(\mathcal{O}_T)$ states that $\mathcal{O}_{Y}$ is a subbundle of $(g\circ h)_{\ast}(\mathcal{O}_T)$ (i.e. the resulting quotient is a vector bundle) since $g\circ h$ is finite and faithfully flat\footnote{The property of being locally free is local for the flat topology. And locally for the flat topology, a finite and faithfully flat morphism $V \to U$ has a section (\cite{Milne}, proof of proposition 2.18, p. 17)}. The morphism $f$ being affine and flat we deduce, from last inclusion, that $f_{\ast}({\mathcal{F}})$ is a subbundle of  $f_{\ast}{(g\circ h)_{\ast}(g\circ h)^{\ast}(\mathcal{F})}=u_{\ast}h^{\ast}g^{\ast}({\mathcal{F}})\simeq (u_{\ast}{\mathcal{O}_{T}})^{\oplus r}$ but $(u_{\ast}{\mathcal{O}_{T}})^{\oplus r}$ is finite, then in particular essentially finite. According to \cite{Nor}, Proposition 3.7, we just need to verify that $f_{\ast}{\mathcal{F}}$ has restricted degree $0$ in order to prove it is essentially finite, but this is a consequence of lemma \ref{lemmaSperanza} since $\mathcal{F}$ has restricted degree $0$. 



\section{The short exact sequence of fundamental group schemes}

\medskip
\label{sez:quattro}
Nori showed in \cite{Nor}, Ch. II, that under the hypothesis of section \ref{sez:Intro}, the category of torsors under finite group schemes over $X$ pointed above $x$ is filtered, and that the fundamental group scheme $\pi _1 (X,x)$ is the projective limit of the groups occurring in these torsors. This led Nori to approach the construction of fundamental group scheme in a different manner: a $k$-scheme pointed at $x \in X(k)$ has a fundamental group scheme based at $x$ if there exists a {\it universal torsor pointed above $x$} that dominates every torsor pointed above $x$ under the action of a finite group scheme (cf. \cite{Nor2}, Chapter II, Definition 1). Then he proves that this is equivalent as saying that the category of torsors under finite group schemes over $X$ pointed above $x$ is filtered (cf. \cite{Nor2}, Chapter II, Poposition 1). This point of view has been generalized by Gasbarri in \cite{Gas}, \S 2 to schemes over Dedekind rings.

As a consequence of  section \ref{sez:tre}, we show here that if $X$ is a proper reduced scheme satisfying the condition $H^0 (X, \mathcal{O}_X)=k$ endowed with a rational point $x \in X(k)$ and $Y \to X$ is a quotient torsor under a finite group scheme $G$ pointed on $y \in Y(k)$ above $x$, then $Y$ has a fundamental group scheme. Moreover $ \pi _1 (X,x)$ and $\pi _1 (Y,y)$ fit in an short exact sequence. This result was obtained independently by Garuti \cite{Gar}, Theorem 3 by another method in the more general situation of relative schemes over Dedekind schemes.

\begin{thm}\label{suite-exacte}
Let $\theta : X \to Spec (k)$ be as before a locally noetherian proper reduced $k$-scheme endowed with a rational point $x \in X(k)$ such that $H^0 (X, \mathcal{O}_X)=k$ and $f : Y \to X$ a quotient torsor under a finite group scheme $G$, pointed on $y \in Y(k)$ over $x$, corresponding to a faithfully flat morphism $\varphi : \pi _1 (X,x) \to G$.

Then $Y$ has a fundamental group scheme based at $y$, $\pi _1 (Y,y)\simeq Ker ( \varphi )$ and we have the following short exact sequence of group schemes:

$$\xymatrix{ 1\ar[r]& \pi _1 (Y,y) \ar[r] & \pi _1 (X,x) \ar[r]^{\varphi }& G \ar[r] & 1\\}$$
\end{thm}

\medskip
\proof Let $f :Y \to X$ be as in the statement of the theorem. It follows from theorem \ref{universel} that there is a unique morphism of torsors $\lambda : \hat X \to Y$ with respect to the morphism $\varphi : \pi _1 (X,x) \to G$ such that $f \circ \lambda = p$ and $ \lambda (\hat x ) = y$. Moreover $\lambda : \hat X \to Y$ is a torsor under $Ker( \varphi )$.

On the other hand let $g : Z \to Y$ be a torsor under a finite group scheme, pointed by $z \in Z(k)$ above $y$. According to corollary \ref{tour-de-torseurs}, $h=f \circ g$ is an essentially finite morphism, and thus is dominated by its Galois closure $\hat h : \hat Z \to X$, which is a quotient torsor itself dominated by $ p : \hat X \to X$, where $p : \hat X \to X$ is the universal torsor of $X$ based at $x$. 

Thus $\lambda :  \hat X \to Y$ is the projective limit of the pointed torsors $g : Z \to Y$; it is the universal torsor of $Y$ based at $y$.
\endproof

\medskip

\begin{cor}
Let $\theta : X \to Spec (k)$ be as before a locally noetherian proper reduced connected $k$-scheme endowed with a rational point $x \in X(k)$ and $f : Y \to X$ an essentially finite morphism pointed at $y \in Y(k)$ above $x$. Assume $H^0 (Y,\mathcal {O}_Y)=k$. Then $Y$ has a fundamental group scheme based at $y$ and $\pi _1 (Y,y) \simeq \pi _1 (X,x) _y$ where $\pi _1 ( X,x) _y$ denotes the stabiliser of $y$ in the natural action of $\pi _1 (X,x) $ on $Y_x$.
\end{cor}

\proof
Let $f' : Y' \to Y$ be a pointed $G'$-torsor on $Y$ under a finite group scheme $G'$. Denote $\hat f : \hat Y \to X$ the Galois closure of $f$ and consider the following diagram, where the upper square is cartesian:

$$\xymatrix{
Z= Y' \times _Y \hat Y  \ar[d]  \ar[r]^{\hskip 0,5cm g'} & Y' \ar[d] ^{f'} \\
\hat Y \ar[r]^g \ar[dr] _{\hat f} & Y\ar[d] ^f \\
&X\\
}$$

Then $Z \to \hat Y$ is a pointed $G'$-torsor, and according to theorem \ref{suite-exacte}, there are unique morphisms of pointed torsors $\lambda : \hat X  \to Z$ and $\mu : \hat X  \to \hat Y$ making the following diagram commutative:
$$\xymatrix{\hat X \ar[r] ^{\lambda \hskip 0,5cm}\ar[dr]_\mu  &
Z= Y' \times _Y \hat Y  \ar[d]  \ar[r]^{\hskip 0,5cm g'} & Y' \ar[d] ^{f'} \\
&\hat Y \ar[r]^g \ar[dr] _{\hat f} & Y\ar[d] ^f \\
&&X\\
}$$
where $ \hat X$ denotes the universal torsor on $X$ based at $x$.

According to theorem \ref{universel}, $g \circ \mu $ is a pointed torsor over $Y$ under the stabiliser $\pi _1 (X,x)_y$ of $y$ in the action of $\pi _1 (X,x)$ on $Y_x$. According to \cite{Nor2}, lemma 1, a $k$-scheme morphism between two torsors $P$ and $P'$ under affine group schemes $G$ and $G'$ on a $k$-scheme $Y$ is a morphism of torsors relative to some unique morphism of groups $G \to G'$ provided that $H^0 (P, \mathcal {O}_P)=k$. Apply this to $g' \circ \lambda $; there exists a unique morphism of groups $\varphi : \pi _1 (X,x) _y \to G'$ such that $g' \circ \lambda $ is a morphism of pointed torsors relative to $ \varphi $. This proves that $g \circ \mu : \hat X \to Y$ is an universal object in the category of pointed torsors under finite group schemes on $Y$. One concludes that $Y$ has a fundamental group scheme and that $\pi _1 (Y,y) \simeq \pi _1 (X,x) _y$. \endproof
\medskip\medskip
\section{An example}

\medskip
We will restrict ourselves to the case of a characteristic $0$ field $k$. With the hypothesis of section \ref{sez:Intro}, we have the classical short exact sequence of Grothendieck \'etale fundamental groups:

$$1 \to \pi _1 ^{\acute{e}t} (X_{Ê\bar k} , \bar x) \to \pi _1^{\acute{e}t} (X, \bar x) \to Gal ( \bar k /k) \to 1$$

where $\bar k$ is an algebraic closure of $k$ and $\bar x$ is the geometric point corresponding to $x$. The rational point $x$ gives rise to a section $s: Gal ( \bar k /k) \to  \pi _1^{\acute{e}t} (X, \bar x)$ and $ \pi _1^{\acute{e}t} (X, \bar x)$ is in this way the semi-direct product of the geometric fundamental group $\pi _1 ^{\acute{e}t} (X_{Ê\bar k} , \bar x)$ by the absolute Galois group of $k$. The section defines an action of $Gal ( \bar k /k)$ on $\pi _1 ^{\acute{e}t} (X_{Ê\bar k} , \bar x)$ and this group endowed with this action can be viewed as a pro-$k$-group scheme. This is the Nori's fundamental group of $X$ based at $x$. In the case of an algebraically closed field $k$ of characteristic $0$, it suffices to remark that finite $k$-group schemes are just finite abstract groups, and that torsors under such finite group schemes are just Galois \'etale coverings (cf. \cite{TS}, Corollary 6.7.20). The general case is explained in Appendix 1, theorem \ref{etale}.
\medskip

The aim of this section is to translate in terms of fundamental short exact sequences of \'etale fundamental groups the Galois closure constructed in section \ref{sez:tre}. Let $f: Y \to X$ be an essentially finite morphism and  let us assume, as usual, that $H^0 (Y,\mathcal{O}_Y)=k$, then $f$ is just  a finite \'etale morphism (Remark \ref{RemEtalEssential}) with $Y$  geometrically connected. If we enumerate the geometric fiber at $\bar x$ as $\{Ê1 , \dots ,d \} $, then the data of the degree $d$ \'etale covering $Y \to X$ is equivalent to a morphism $\psi: \pi _1^{\acute{e}t} (X,\bar x) \to S_d$. This morphism factors as $\psi = \theta \circ \Phi $, where $ \Phi:  \pi _1^{\acute{e}t} (X,\bar x) \to G$, $G$ is the image of $\psi $ and $\theta $ the inclusion $G \subset S_d$. The surjective morphism $\Phi$ corresponds to the Galois closure $Z \to X$ of $f: Y \to X$ in the sense of Galois theory. Its restriction to $\pi _1 ^{\acute{e}t} (X_{\bar k }, \bar x)$ factors through a surjective morphism $ \varphi: \pi _1 ^{\acute{e}t} (X_{\bar k}, \bar x) \to H$.  

We have the following commutative diagram where all lines and the first and third columns are exact:

$$\xymatrix{
1 \ar[r] & \pi_ 1^{\acute{e}t} ( X_{\bar k } , \bar x) \ar[d]^\varphi \ar[r] & \pi _1^{\acute{e}t}  (X , \bar x) \ar[d] ^\Phi \ar[r] & Gal( \bar k /k) \ar[d] \ar[r] & 1 \\
1 \ar[r]& H \ar[r]  \ar[d]& G \ar[r] \ar[d]^\theta & Gal( L/k) \ar[r]\ar[d] &1\\
&1&S_d&1&\\
}$$

where $L=H^0 (Z, \mathcal{O}_Z)$ is a finite Galois extension of $k$. 

The section $s$ attached to the rational point $x$ induces a section $s_L: Gal ( L/k) \to G$ of the second exact sequence 
$$\xymatrix{
1 \ar[r]& H \ar[r] & G \ar[r]  & Gal( L/k) \ar[r]&1\
}$$
and as we have seen, the first exact sequence with the section $s$ is equivalent to the data of the fundamental group scheme $\pi _1 (X,x)$ and the second exact sequence with the section $s_L$ is equivalent to the data of a $k$-group scheme $H_k$ such that $ H_k \times _k \bar k \simeq H$. The morphism $ (\varphi , \Phi )$ corresponds to a morphism $\pi _1 (X,x) \to H_k$, and thus according to theorem \ref{thReferee}, to a torsor $T \to X$ pointed above $x$. When $Y$ has a rational point $y$ in the fiber of $x$, this torsor is the Galois closure of the essentially finite morphism $ f : YÊ\to X$.
\medskip

It is an exercise to describe this Galois closure in terms of classical Galois theory and descent theory : geometrically the above diagram corresponds to the following diagram:

$$\xymatrix{
Z \ar[r] \ar[d] & Y \ar[d] ^f \\
X_L \ar[d] \ar[r] & X \ar[d] \\
Spec (L) \ar[r] & Spec (k)\\
}$$
where $Z$ is a Galois geometrically connected \'etale cover of $X_L$ of group $H$. The section $s_L$ induces descent data from $L$ to $k$ of the Galois cover $Z \to X_L$ in a cover $p: T \to X$ as well as of the constant group $H$ in a $k$-group scheme $H_k$ (the Hopf algebra of the constant group $H$ over $L$ is $L^H$, whereas the Hopf algebra of $H_k$ is the sub-algebra $(L^H)^{Gal (L/k)}$ of fixed elements under the natural action of $Gal (L/k)$). Moreover these descent data are compatible with the right multiplication of $H$ on itself identified with the fiber at $\bar x$ which commutes with the action of the fundamental group and thus induces a right action of $H_k$ on $T$.
\medskip

One can describe explicitly these actions in terms of the preceding diagrams. First the fiber of $Z \to X_L$ at $\bar x$ is identified to $H$ and the action of an element $\gamma $ of $\pi _1 (X_{\bar k} , \bar x)$ on this fiber is the left multiplication of $\varphi (\gamma ) \in H$ on $H$ through this identification, whereas the fiber of $Z \to X$ at $x$ is identified with $G$. Secondly the action of an element $\sigma \in Gal( \bar k /k)$ on the fibers at $x$ - that of $Z\to X$ identified with $G$ and that of $Y\to X$ identified with $\{Ê1 , \dots , d \} $ - is the conjugation by $\Phi \circ s ( \sigma )$ (resp. $\theta \circ \Phi \circ s ( \sigma )$).
\medskip

The \'etale cover $p: T \to X$, which is in general not Galois, corresponds to the morphism $\Psi: \pi _1 ( X , \bar x) \to S_H$ defined in the following way:
$$\forall \gamma \in \pi _1^{\acute{e}t} ( X_{\bar k}, \bar x), \forall \sigma \in Gal( \bar k /k), \forall h \in H \quad \Psi ( \gamma s(\sigma )).h= \varphi (\gamma ) \Phi (s( \sigma )) h \Phi (s( \sigma )) ^{-1}$$
which is an action clearly extending the action of $\pi _1 ^{\acute{e}t} ( X_{\bar k } , \bar x )$ on $H$ via $\varphi $. The \'etale cover $p: T \to X$ which is a torsor under $H_k$ is the Galois closure of $f: Y \to X$ defined in section \ref{sez:tre}. 
\medskip

With the tools introduced here one can check that $p: T \to X$ factors through $f: Y \to X$ if $Y$ has a $k$-rational point $y$. Indeed one can index the points of the geometric fiber at $x$ such that $y$ corresponds to $1$, which will be fixed under the action of $Gal ( \bar k/k)$, i.e.

$$\forall \sigma \in Gal ( \bar k /k) \quad (\theta \circ \Phi \circ s ( \sigma )).1=1$$

Let us define a map $H \to \{Ê1 , \dots , d \} $ by the formula

$$F: h \to (\theta  (h)).1$$

As $Y\to X$ is geometrically connected, the action of $H$ on $\{Ê1, \dots , d \} $ is transitive, and thus $F$ is a surjective map. It is compatible with the actions of $ \pi _1 ( X, \bar x)$ on $H$ and $\{Ê1, \dots , d \}Ê$, and thus defines an unique morphism of covers $g: T \to Y$ over $X$. This is indeed a consequence of the following computation:

$$F((\gamma s(\sigma )) .h)=F(\varphi ( \gamma ) \Phi ( s(\sigma )) h \Phi (s(\sigma ))^{-1})=$$$$=( \theta \circ \varphi ( \gamma )\hskip 0,2cm  \theta \circ \Phi (s( \sigma )) \hskip 0,2cm  \theta (h) \hskip 0,2cm \theta \circ \Phi (s( \sigma ) )^{-1}) .1=( \theta  \circ \varphi ( \gamma )\hskip 0,2cm  \theta \circ \Phi (s( \sigma ) )\hskip 0,2cm  \theta  (h) ).1$$

In the other hand

$$\gamma s(\sigma ) .F(h)= \gamma s(\sigma ). (\theta  (h).1)= \theta ( \Phi ( \gamma s( \sigma ) )\hskip 0,2cm h ).1= \theta (  \varphi (\gamma ) \hskip 0,2cm \Phi ( s( \sigma )) \hskip 0,2cm h ) .1=$$$$= (\theta  \circ \varphi ( \gamma )\hskip 0,2cm  \theta \circ \Phi (s( \sigma ))  \hskip 0,2cm \theta  (h)) .1$$

Finally remark that the torsor $p: T\to X$ is a Galois cover of $X$ if and only if the action of $Gal (\bar k /k)$ on $H$ is trivial or equivalently the Nori Galois group $H_k$ is the constant group $H$. In this case the the group $G$ is isomorphic to the direct product
$$G \simeq H \times Gal ( L/k)$$

\medskip

Consider the particular case where $f: Y \to X$ is itself a geometrically connected torsor under a finite group scheme $H_0$ and suppose that $Y(k) \not= \emptyset $. According to corollary \ref{cas-des-torseurs} the group-scheme $H_k$ defined above is isomorphic to $H_0$ and the morphism $g: T \to Y$ defined above is an isomorphism of torsors. We have the following cartesian diagram 

$$\xymatrix{
Z\simeq Y_L \ar[r] \ar[d] & Y \ar[d] ^f \\
X_L \ar[d] \ar[r] & X \ar[d] \\
Spec (L) \ar[r] & Spec (k)\\
}$$

Let us consider another model $f': Y' \to X$ over $k$ of the Galois cover $Z\simeq Y_L \to X_L$. It corresponds to another section $s'$ of the short exact sequence

$$\xymatrix{
1 \ar[r]& H\ar[r] & G \ar[r] & Gal(L/k) \ar[r] & 1\\
}$$

where $H= H_0 \times _{Spec (k)} Spec (L)$ is a constant group and $G= Gal ( Z/X)$. The group $H$ with the action of $Gal ( L/k)$ defined by $s'$ defines a $k$-group scheme $H'$ and $Y' \to X$ is a torsor under $H'$. 

If $Y'(k) \not= \emptyset $, the argument used for $Y \to X$ applies to $Y' \to X$, and one concludes that the torsor $Y' \to X$ is isomorphic to $T \to X$, and thus to $Y \to X$. We have shown the following statement:

\begin{prop}

Let $f: Y \to X$ be a geometrically connected torsor defined over $k$ under a $k$-group scheme $H_0$, $L/k$ the smallest Galois extension of $k$ such that $H=H_0 \times _{Spec (k) } Spec (L)$ is constant. Suppose that $Y_x (k) \not= \emptyset $. Then there is a one to one correspondence between classes of conjugation of sections of the short exact sequence
$$\xymatrix{
1 \ar[r]& H\ar[r] & G \ar[r] & Gal(L/k) \ar[r] & 1\\
}$$
and classes of isomorphism of $k$-models $Y' \to X$ of $Y_L \to X_L$. Both sets are parametrized by $H^1 ( k, H_0)$ pointed by the section $s$ attached to the rational point $x \in X(k)$. This section corresponds to the unique (up to $k$-isomorphim) model $Y \to X$ such that $Y_x (k ) \not= \emptyset $.
\end{prop}

\medskip
\section{Appendix 1}

\subsection{Gerbes, groupoids, and short exact sequences}\label{gerbes}

The general tannakian duality statement says that a tannakian category is equivalent to the category of representations of the gerbe of its fiber functors, or equivalently to the category of representations of some $k$-groupoid. The correspondence between gerbes over $k$ and $k$-groupoids acting transitively on a $k$-scheme is explained in \cite{Del}. Let us recall it briefly.

\medskip
\begin{defi}
A $k$-groupoid $G$ acting on a $k$-scheme $S$ is a $k$-scheme $G$ given with a $k$-morphism $(t,s) : G \to S \times _k S$ (target and source) and a product morphism $$m : G {\times _{{}^sS^t} } G \to G$$ over $S\times _k S$, a unit element morphism $e: S \to G$ over the diagonal $S \to S\times _k S$, and an inverse element morphism $i : G \to G$ over the morphism $S \times _k S \to S\times _k S$ which maps $(s_1,s_2)$ to $(s_2,s_1)$ ; these morphism must satisfy the commutativity of diagrams expressing associativity of $m$, and properties of unit element and of inverse elements.
\end{defi}

\medskip




If $(t,s) : G \to S \times _k S$ is a $k$-groupoid acting on $S$, and $f : T \to S$ is a $k$-morphism of $k$-schemes, the pull back $G_T$ of $G$ is a $k$-groupoid acting on $T$ :

$$\xymatrix{
G_T \ar[r] \ar[d] & G \ar[d] ^{s,t} \\
T \times _k T \ar[r] ^{f,f} & S\times _k S \\
}$$




\begin{defi}
The $k$-groupoid $ G \to S \times _k S $ acts {\it transitively} on $S$ if there is a {\it fpqc}-covering $ T\to S\times _k S $ such that $ G_T (T) \not= \emptyset $. 


\end{defi}
\medskip

\begin{defi}
A gerbe $\cal G$ over $S$ for the {\it fpqc}-topology is a stack over $S$ for the {\it fpqc}-topology such that

\begin{enumerate}

\item ${\cal G}$ is locally non-empty: there is a covering of $S$ by $(U_i)_i$ such that ${\cal G} (U_i) \not= \emptyset $

\item two objects are locally isomorphic: if $\xi $ and $\xi' $ are objects of ${\cal G} (T)$, where $T \to S$, there is a covering $(T_j) _j$ of $T$ such that, for all $j$, $\xi _{| T_j}\simeq \xi '_{| T_j}$

\end{enumerate}
\end{defi}
\medskip

Given a gerbe $\cal G$ over a field $k$ and a section $\omega $ over some $k$-scheme $X$, one defines the groupoid $\Gamma _{X, {Ê\cal G} , \omega }= \underline {Aut } ( \omega) \to X \times _k X$, representing the functor associating to any morphism $(b,a) : T \to X\times _k X $ the set $Isom _T^{Ê\cal G} ( a^\ast \omega , b^\ast \omega )$. If $u : Y \to X$ is a $k$-morphism, one has the following formula 

$$u^\ast \Gamma _{ X, {Ê\cal G} , \omega }= \Gamma _{ Y, {Ê\cal G} , u ^\ast \omega }$$

In the other direction, given a groupoid $ \Gamma \to X \times _k X$ acting transitively on the $k$-scheme $X$, one defines a fiber category ${\cal G} ^0 _{X, \Gamma }$ whose objects are $k$-morphisms $T \to X$ and where the morphisms from $a : T \to X$ to $b : T \to X$ are $(a,b)^\ast \Gamma $. Let ${\cal G } _{X, \Gamma }$ be the stack associated to ${\cal G} ^0 _{X, \Gamma }$. The fact that the groupoid acts transitively on the $k$-scheme $S$ means that any two objects of  ${\cal G } _{X, \Gamma }$ are locally isomorphic for the $fpqc$-topology, and thus ${\cal G } _{X, \Gamma }$ is a gerbe. 

Any morphism $ u : Y \to X$, where $Y$ is a non empty $k$-scheme induces an equivalence 

$${\cal G} _{Y, u ^\ast \Gamma } \simeq {\cal G }Ê_{X, \Gamma }$$

\medskip
One can define also similar notion for the \'etale topology in place of $fpqc$-topology. We are considering here the case of the \'etale topology on $Spec (k)$ and we are going to explain that a gerbe ${\cal G }$ over $Spec (k)_{\acute{e}t}$ (or equivalently a $k$-groupoid acting on $Spec (\bar k)$) is equivalent to a short exact sequence built from a section $ \omega  \in {\cal G } ( \bar k )$:

$$1 \to Aut_{\bar k} ^{\cal G} ( \omega ) \to \Pi \to Gal ( \bar k /k) \to 1$$

\noindent where $\bar k$ denotes the separable closure of $k$ and $\Pi $ is some profinite group we are defining as follows.

As explained above, from $\omega \in {\cal G} ( \bar k)$ one can define the groupoid $(s,t) : \Gamma = \underline{Aut} ( \omega ) \to Spec ( \bar k) \times _k Spec ( \bar k )$ acting transitively on $Spec ( \bar k)$. Call $ \Delta = Spec ( \bar k) \times _k Spec ( \bar k)$ the trivial groupoid acting on $Spec ( \bar k)$. And denote by $s : \Gamma _1 \to Spec ( \bar k)$ (resp. $pr_1 :  \Delta _1\to Spec ( \bar k )$) the schemes $\Gamma $ (resp. $ \Delta $) endowed with the map $s$ (resp. $ pr_1$) over $Spec ( \bar k)$.

The set $\Delta _1 ( \bar k)$ is canonically in bijection with $Gal ( \bar k /k)$ and an element $\gamma \in \Gamma _1 ( \bar k)$ which maps to $ \sigma \in Gal ( \bar k /k) \simeq \Delta _1 ( \bar k)$ is an element of $Isom _{Ê\bar k } ^{\cal G}( \omega , \sigma ^\ast \omega )$. There is a natural group structure on $\Gamma _1 ( \bar k)$ compatible with the map $ \Gamma _1 ( \bar k ) \to \Delta _1 ( \bar k)$ and the group structure on $ \Delta _1 ( \bar k ) \simeq Gal ( \bar k /k )$: it is defined in the following manner: if $\sigma , \tau \in Gal ( \bar k /k)$, $\gamma \in Isom _{Ê\bar k } ^{\cal G}( \omega , \sigma ^\ast \omega )$, $\delta \in Isom _{Ê\bar k } ^{\cal G}( \omega , \tau ^\ast \omega )$, then

$$\gamma \ast \delta = {}^\sigma \delta \gamma \in Isom _{Ê\bar k } ^{\cal G}( \omega , (\tau \sigma ) ^\ast \omega )$$

\noindent thus the map $ \Gamma _1 ( \bar k ) \to \Delta _1 ( \bar k)$ is a group homomorphism whose kernel is $Aut _{ \bar k } ^{\cal G}( \omega )$. Thus one gets the following short exact sequence 

$$(1) \quad 1 \to Aut _{Ê\bar k }^{\cal G} ( \omega )\to \Gamma _1 ( \bar k ) \to Gal ( \bar k /k ) \to 1 .$$

Let $k \subset K \subset \bar k$ be an algebraic extension of $k$; one can pull the short exact sequence $(1)$ by the morphism $Gal ( \bar k /K) \to Gal ( \bar k /k)$. A section $s$ of this exact sequence is the data for all $ \sigma \in Gal ( \bar k /K)$ of an isomorphism $ \varphi _\sigma : \omega \to \omega ^{\sigma }$ satisfying Weil cocycle conditions or equivalently descent data from $\bar k $ to $K$ for $\omega$. It is thus a section $\tilde s$ of the gerbe $\cal G$ on $Spec (K)$. 

Given two sections $s$ and $t$ of the short exact sequence $(1)$ on $Spec (K)$, the morphisms from $\tilde s$ to $\tilde t$ in ${\cal G } _L$, where $K\subset L\subset \bar k$ is a finite extension of $K$, are the automorphisms of $\omega$ which are compatible with the descent data, that is the elements $\gamma \in Aut_{\bar k} ^{\cal G}( \omega)$ satisfying $$ \forall \sigma \in Gal ( \bar k /L) \quad \gamma \star s( \sigma )= t(\sigma ) \star \gamma $$

\noindent which expresses the commutativity of the following diagrams

$$\xymatrix{
\bar x ^\star \ar[d]_{s( \sigma ) } \ar[r] ^\gamma & \bar x ^\star \ar[d] ^{t( \sigma ) }\\
({}^\sigma \bar x ) ^\star \ar[r]^{{}^\sigma \gamma}& ({}^\sigma \bar x ) ^\star \\
}$$
\medskip
We have defined a fully faithful functor $s \to \tilde s$ from the category of sections of the short exact sequence $ (1)$ to the gerbe $ \cal G$. This functor is essentially surjective: let $ \rho $ be a section of $ \cal G$ on an extension $K$ of $k$, $k\subset K \subset \bar k$, then $\rho _{\bar k}$ is isomorphic over $\bar k$ to the section $\omega$. And then $ \rho $ is defined from $ \omega $ through descent data. As we have seen above one can associate to these descent data a section of the short exact sequence $(1)$. We have proved the following statement:

\medskip

\begin{prop}\label{sections-foncteurs-fibres}
Any section $s$ of the short exact sequence $(1)$ on a finite extension $K$ of $k$ gives rise to descent data from $\bar k $ to $K$ for $ \omega $ and then to a section $ \tilde s$ of the gerbe $\cal G$ on $K$. The functor $ s \to \tilde s$ is an equivalence of gerbes between the gerbe of sections of the short exact sequence $(1)$ and the gerbe $ \cal G$.
\end{prop}

\medskip

When the gerbe $\cal G$ is {\bf neutral}, i.e. possesses a section over $Spec (k)$, choose such a section $\xi \in {Ê\cal G} ( Spec (k))$ and $\omega = \xi _{Ê\bar k}$. Then the descent data from $\bar k $ to $k$ defining $\xi$ from $\omega$ give rise to a section $s$ of the short exact sequence:
$$s ( \sigma ) : \xi _{Ê\bar k} \simeq {} ^\sigma \xi _{Ê\bar k} =\xi _{Ê\bar k} $$

and for $\lambda \in Aut_{\bar k}^{\cal G} (\xi _{Ê\bar k} )$, ${}^\sigma \lambda = s( \sigma ) \star \lambda \star s ( \sigma )^{-1}$. One gets the following statement: 

\begin{prop}\label{groupe} Suppose the gerbe $\cal G$ being {\bf neutral} and choose a section $\xi \in {Ê\cal G} ( Spec (k))$. The $k$-group $Aut ^{\cal G} ( \xi )$ is then defined by the abstract group $Aut_{\bar k}^{\cal G} ( \xi _{Ê\bar k})$ endowed with the action of $Gal ( \bar k/k)$ by conjugation through $s$.
\end{prop}

\subsection{Nori's fundamental group and Grothendieck's fundamental group}
 
We assume here that $ch(k)=0$. Let $X$ be as usual a proper reduced and connected $k$-scheme with a rational point $x \in X( k)$. Denote by $ \bar x$ the geometric point corresponding to $x$ and by  $ k \hookrightarrow \bar k$ the embedding of fields. This geometric point $\bar x$ gives rise to a fibre functor $\bar x ^\ast $ from the Galois category of \'etale covering of $X$ to the category of sets (identified to the category of finite \'etale $\bar k $-schemes). On the other hand $x$ gives rise to a neutral fiber functor $\omega _x $ of the tannakian category $EF(X)$ to the category of finite $k$-vector spaces. 

In this situation, the arithmetic and geometric \'etale fundamental groups are well defined and fit in the fundamental short exact sequence 

$$(2) \quad 1 \to \pi _1 ^{\acute{e}t} (X_ {Ê\bar k} , \bar x) \to \pi _1^{\acute{e}t} (X, \bar x) \to Gal ( \bar k /k) \to 1.$$

\noindent The rational point $x \in X(k)$ defines a splitting of this short exact sequence, and thus an action of $Gal ( \bar k /k)$ over the abstract profinite group $\pi _1 ^{\acute{e}t} (X_ {Ê\bar k} , \bar x)$.

\begin{thm}\label{etale} The Nori's fundamental group scheme $Aut ^{\cal G} (\omega _x )= \pi _1 (X,x)$ is the $k$- scheme defined by $\pi _1 ^{\acute{e}t} (X_ {Ê\bar k} , \bar x)$ endowed with the action defined above.
\end{thm}
\medskip

\proof In the gerbe $\cal G$ of fibre functors of the tannakian category $EF(X)$, for any fibre functor $\omega $, $Aut ^{\cal G} ( \omega ) = Aut ^\otimes ( \omega )$, where this notation represents the group of automorphisms of $\omega $ compatible with the tensor product. The fact that Nori's fundamental group $Aut ^{\cal G} (\omega _x )$ is the $k$-group defined by $\pi _1 ^{\acute{e}t} (X_ {Ê\bar k} , \bar x)$ endowed with the action defined above is a consequence of Proposition \ref{groupe} together with the following proposition (see also theorem 4.4. of \cite{EsnHai1}) \endproof

\begin{prop}

There are isomorphisms making the following diagram commutative:

$$\xymatrix{
1 \ar[r] & Aut^{\cal G} ( \omega _x) \ar[d] ^\simeq \ar[r] & \Gamma _1 ( \bar k ) \ar[r] \ar[d] ^\simeq & Gal ( \bar k /k ) \ar[r] \ar[d] ^= & 1 \\
1 \ar[r] & \pi _1 ^{\acute{e}t} ( X_ {Ê\bar k} , \bar x)  \ar[r] &\pi _1 ^{\acute{e}t} ( X_ {} , \bar x)\ar[r] & Gal ( \bar k /k ) \ar[r]  & 1 \\
}$$
\noindent where the first row is  the exact sequence (1) of section \ref{gerbes} and the second row the fundamental exact sequence (2).
\end{prop}

\proof For all $\sigma \in Gal ( \bar k /k)$ and for any \'etale covering $h: Y \to X$,  we have a cartesian diagram

$$\xymatrix{
Y_{Ê\bar k} ={}^\sigma Y_{Ê\bar k} \ar[d]^{{}^\sigma h_{Ê\bar k} } \ar[r] ^{\beta _\sigma }& Y_{\bar k} \ar[d] ^{h_{Ê\bar k}} \\
X_{Ê\bar k}  \ar[r] ^{\alpha _\sigma } \ar[d] & X_{Ê\bar k}\ar[d]  \\
Spec ( \bar k) \ar[r] ^{Ê\tilde \sigma } & Spec ( \bar k)\\
}$$
which defines ${}^\sigma Y_{Ê\bar k}$ and ${}^\sigma h_{Ê\bar k} $.

The restrictions of $ \beta _\sigma $ to the fibers of $\bar x$ and ${}^\sigma \bar x $ induce maps between finite sets $$\beta _\sigma : ({}^\sigma \bar x )^\star ({}^\sigma Y) \to \bar x ^\star (Y)$$

They define a natural transformation that we still denote $\alpha _\sigma $ from ${}^\sigma \bar x^\star $ to $ \bar x ^\star $. To $ \gamma \in Isom ^{\cal G} (\omega _{\bar x} , {}^\sigma \omega _{\bar x})\subset \Gamma _1 ( \bar k)$ let us associate $\tilde \gamma : \bar x ^\ast \Rightarrow {}^\sigma \bar x ^\ast $ and define $\Phi ( \gamma ) = \alpha _\sigma \circ \tilde \gamma \in \pi _1 ( X, \bar x)$\footnote{The category of \'etale finite covering of $X$ can be identified to a subcategory of $EF(X)$ by the functor which sends a finite \'etale covering $f: Y \to X$ to $f_\ast \mathcal{O}_Y$. We are identifying the restriction to this subcategory of $\omega _{ x, \bar k}$ with $\bar x ^\ast $ }. The following commutative diagram proves that $\Phi $ is a group homomorphism:

$$\xymatrix{
({}Ê{}^{Ê\sigma \tau }\bar x )^\star ({}^{\sigma \tau }Y) \ar[r]^{Ê{}^\sigma \beta _{\tau }} & ({}^{Ê\sigma }\bar x )^\star ({}^{\sigma }Y)\ar[r] ^{\beta _\sigma }& ({}^{}\bar x )^\star ({}^{}Y)\\
& ({}^{Ê\sigma }\bar x )^\star ({}^{\sigma }Y)\ar[lu]^{{} ^\sigma {\tilde \delta }} \ar[u]^{{}^\sigma \Phi ( \delta )}\ar[r] ^{Ê\beta _\sigma } &( {}^{Ê }\bar x )^\star ({}^{ }Y)\ar[u] ^{\Phi ( \delta )}\\
&&( {}^{Ê }\bar x )^\star ({}^{}Y)\ar[lu]^{\tilde \gamma }\ar[u] ^{\Phi ( \gamma ) }\\
}$$
\medskip
To verify that $\Phi $ is an isomorphism, il suffices to check that the diagram
$$\xymatrix{
1 \ar[r] & Aut^{\cal G} ( \omega _x) \ar[d] ^{\Phi _{| Aut^{\cal G} ( \omega _x)}} \ar[r] & \Gamma _1 ( \bar k ) \ar[r] \ar[d] ^\Phi & Gal ( \bar k /k ) \ar[r] \ar[d] ^= & 1 \\
1 \ar[r] & \pi _1 ^{\acute{e}t} ( X_ {Ê\bar k} , \bar x)  \ar[r] &\pi _1 ^{\acute{e}t} ( X_ {} , \bar x)\ar[r] & Gal ( \bar k /k ) \ar[r]  & 1 \\
}$$
is commutative and that the two maps $\Phi _{|Aut ^{\cal G} ( (\omega _ x ) ) }: Aut^{\cal G} ( \omega _x) \to \pi _1 ( X_{Ê\bar k } , \bar x)$ and $\Delta _1 ( \bar k ) \to Gal (\bar k /k)$ are isomorphisms. It is obvious for the second one. As for the first one it suffices to notice that $Aut^{\cal G} ( \omega _x)$ is the Nori's fundamental group of $X_{Ê\bar k }$ based at $ \bar x$, which is known to be isomorphic to the projective limit of finite $\bar k $-group schemes occurring in finite torsors $Y\to X_{ \bar k}$, which is equivalent to finite \'etale Galois covering of $X_{Ê\bar k}$. Thus $Aut^{\cal G} ( \omega _x)$ is isomorphic to $\pi _1 ( X_{Ê\bar k } , \bar x)$.

To check that the diagram is commutative we only have to check that the right square is commutative. The morphism $ \pi _1^{\acute{e}t} ( X, \bar x) \to Gal ( \bar k /k)$ is associated in the Galois theory with the functor which sends any finite \'etale $k$-algebra $k \subset K$ to the purely arithmetic  covering $X \times _{ÊSpec ( k)}ÊSpec (K) \to X$.

Let $k \subset K$ be a finite \'etale $k$-algebra. The structural morphism $ \bar x ^\star (X_K) \to Spec ( \bar k)$ can be identified canonically to $Spec (K \otimes _k \bar k) \simeq Spec ( \bar k ^{S_K}) \to Spec ( \bar k)$, where $S_K$ is the set of $k$-embeddings of $K$ in $\bar k$, corresponding to the diagonal morphism $\bar k \hookrightarrow \bar k ^{S_K}$. In particular it does not depend on the $\bar k $-point $\bar x$.

Let $ \gamma $ be in $ Isom ^{\cal G} ( \omega _{Ê\bar x}   , \sigma  ^\ast \omega _{Ê\bar x}  )\subset \Gamma _1 ( \bar k)$ where $ \sigma \in Gal ( \bar k /k)$. When we restrict $\gamma $ to the full subcategory $\cal T $ of $EF(X)$ whose objects are $\mathcal{O}_{X_K}$, where $k \hookrightarrow K $ runs among finite \'etale $k$-algebras (or more generally finite $k$-vector spaces), we get a tensor automorphism of the trivial fibre functor extended to $ \bar k$ from the category $EF( Spec (k))$. It is easy to check that the Nori fundamental group of $Spec (k)$ is trivial, and thus, the restriction of $\gamma $ to $\cal T$ is trivial. 

On the other hand, when we restrict the natural transformation $ \alpha _\sigma $ to objects of the form $ X _K \to X$, where $K$ is a finite \'etale $k$-algebra, $\sigma $ induces $1_K \otimes \sigma : K \otimes _k \bar k \to K \otimes _k \bar k$, and modulo the isomorphism $ K \otimes _k \bar k \simeq \bar k ^{S_K}$, the isomorphism $\bar k ^{S_K} \to \bar k ^{S_K}$ given by the following formula:

$$(\star ) \quad ( \lambda _\varphi ) _{Ê\varphi \in S_K} \to (\sigma ( \lambda _{\sigma ^{-1}\varphi })) _{Ê\varphi \in S_K} $$

Finally, the restriction of $\Phi ( \gamma )= \alpha _\sigma \circ \tilde \gamma $ to objects of the form $ X_K \to X$ is given by the formula $ ( \star )$, which corresponds on the set $S_K$ of $ \bar k$ points of $ \bar k ^{S_K}$ to the map

$$\varphi \to \sigma \circ \varphi $$

We have checked that the image of $\Phi (\gamma )\in \pi _1 ( X , \bar x)$ in $Gal ( \bar k /k)$ is $ \sigma \in Gal ( \bar k /k)$ as expected. \endproof
\indent \textbf{Acknowledgements}. Thanks to Niels Borne whose remarks in our numerous discussions were essential in this work. Thanks also to Marco Garuti for his very useful comments. Finally we would like to thank the referee for his numerous suggestions which improved this article.
\bigskip 

\medskip

\scriptsize

\begin{flushright} Marco Antei\\ Laboratoire Paul Painlev\'e, U.F.R. de
Math\'ematiques\\Universit\'e des Sciences et des Technologies de
Lille 1\\ 59 655 Villeneuve d'Ascq\\E-mail: \texttt{antei@math.univ-lille1.fr}\\ \texttt{marco.antei@gmail.com}\\
\end{flushright}

\begin{flushright}Michel Emsalem\\ Laboratoire Paul Painlev\'e, U.F.R. de
Math\'ematiques\\Universit\'e des Sciences et des Technologies de
Lille 1\\ 59 655 Villeneuve d'Ascq\\E-mail: \texttt{emsalem@math.univ-lille1.fr}\\\end{flushright}

\end{document}